\documentclass{amsart}
\usepackage{amsmath}
\usepackage{amssymb}
\usepackage{amsthm}
\usepackage[top=3cm,left=3.5cm,right=3.5cm,bottom=3cm]{geometry} 
\usepackage{tgtermes}
\usepackage[T1]{fontenc}
\usepackage{enumerate}
\usepackage[all]{xy}
\usepackage{color}
\usepackage[dvipdfmx]{hyperref}
\numberwithin{equation}{section}

\theoremstyle{plain}
\newtheorem{theorem}{Theorem}[section]
\newtheorem{corollary}[theorem]{Corollary}
\newtheorem{proposition}[theorem]{Proposition}
\newtheorem{lemma}[theorem]{Lemma}
\newtheorem{sublemma}[theorem]{Sublemma}
\theoremstyle{definition}
\newtheorem{definition}[theorem]{Definition}
\theoremstyle{remark}

\newtheorem{example}[theorem]{Example}
\newtheorem{claim}[theorem]{Claim}

\newcommand{\mbb}[1]{\mathbb{#1}}
\newcommand{\mbf}[1]{\mathbf{#1}}
\newcommand{\mcal}[1]{\mathcal{#1}}

\newcommand{\msf}[1]{\mathsf{#1}}
\newcommand{\id}{\mathrm{id}}

\DeclareMathOperator{\sr}{sr}
\DeclareMathOperator{\GL}{GL}
\DeclareMathOperator{\dom}{dom}
\DeclareMathOperator{\diag}{diag}
\DeclareMathOperator{\coker}{coker}

\DeclareMathOperator{\Tor}{Tor}
\DeclareMathOperator{\SL}{SL}
\DeclareMathOperator{\Ad}{Ad}
\DeclareMathOperator{\Tot}{Tot}
\DeclareMathOperator{\hofib}{hofib}

\title[Homology pro stability]{Homology pro stability for Tor-unital pro rings}
\author{Ryomei Iwasa}
\address{Graduate School of Mathematical Sciences, the University of Tokyo, 3-8-1 Komaba, Meguro-ku, Tokyo, 153-8914 Japan.}
\email{ryomei@ms.u-tokyo.ac.jp}
\thanks{}

\begin{document}

\maketitle

\begin{abstract}
Let $\{A_m\}$ be a pro system of associative commutative, not necessarily unital, rings.
Assume that the pro systems $\{\Tor^{\mbb{Z}\ltimes A_m}_i(\mbb{Z},\mbb{Z})\}_m$ vanish for all $i>0$.
Then we prove that the sequence
\[
	\{H_l(\GL_n(A_m))\}_m \to \{H_l(\GL_{n+1}(A_m))\}_m \to \{H_l(\GL_{n+2}(A_m))\}_m \to \cdots
\]
stabilizes up to pro isomorphisms for $n$ large enough than $l$ and the stable range of $A_m$'s.
\end{abstract}

\tableofcontents

\section{Introduction}

\subsection{}
The homology stability for general linear groups is a simple but deep question on homological algebra.
Let $R$ be an associative unital ring.
We consider the general linear groups $\GL_n(R)$ of $R$ and their sequence
\[
	\GL_n(R) \hookrightarrow \GL_{n+1}(R) \hookrightarrow \GL_{n+2}(R) \hookrightarrow \cdots,
\]
where each embedding is given by sending $\alpha$ to $\left(\begin{smallmatrix} \alpha & 0 \\ 0 & 1 \end{smallmatrix}\right)$.
The question is whether the induced sequence of the integral group homology
\[
	H_l(\GL_n(R)) \to H_l(\GL_{n+1}(R)) \to H_l(\GL_{n+2}(R)) \to \cdots
\]
stabilizes for $n$ large enough than $l$.
There have been many works on this problem, and the most striking result was obtained by Suslin.
\begin{theorem}[Suslin \cite{Su82}]
Let $R$ be an associative unital ring and $l\ge 0$.
Then the canonical map
\[
	H_l(\GL_n(R)) \to H_l(\GL_{n+1}(R))
\]
is surjective for $n\ge \max(2l,l+\sr(R)-1)$ and bijective for $n\ge \max(2l+1,l+\sr(R))$, where $\sr(R)$ is the stable range of $R$.
\end{theorem}

Things become much harder and interesting if we consider \textit{non-unital} rings.
Then the homology stability is strongly related with the $K$-theory excision and the Tor-unitality.
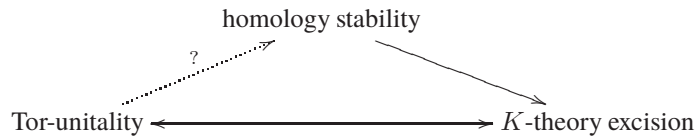
\[
\xymatrix{
	& \text{homology stability} \ar[rd] & \\
	\text{Tor-unitality} \ar@{<->}[rr] \ar@{.>}[ru]^? & & \text{$K$-theory excision}
}
\]
Let $R$ be an associative unital ring and $A$ a two-sided ideal of $R$.
We define the $n$-th relative $K$-group by
\[
	K_n(R,A) := \pi_n \hofib(B\GL(R)^+ \to B\GL(R/A)^+).
\]
We say that \textit{$A$ satisfies the $K$-theory excision} if, for every unital ring $R$ which contains $A$ as a two-sided ideal and for every $n\ge 1$, the canonical map
\[
	K_n(\mbb{Z}\ltimes A,A) \xrightarrow{\sim} K_n(R,A)
\]
is an isomorphism.
It is well-known that the $K$-theory excision fails in general.
However, if the homology $H_l(\GL_n(A))$ stabilizes for $n$ large enough, then $A$ satisfies the $K$-theory excision.
Such being the case, the homology stability for non-unital rings fails in general, even if the stable range of $A$ is finite.

On the other hand, in \cite{Su95}, Suslin has completely determined the obstruction to the $K$-theory excision:
An associative ring $A$ satisfies the $K$-theory excision if and only if $A$ is \textit{Tor-unital}, i.e.\ $\Tor^{\mbb{Z}\ltimes A}_i(\mbb{Z},\mbb{Z})=0$ for all $i>0$.
Hence, we may hope that Tor-unital rings satisfy the homology stability.
Again, Suslin has given a partial solution.
\begin{theorem}[Suslin \cite{Su96}]\label{stab:Su96}
Let $A$ be a Tor-unital $\mbb{Q}$-algebra, $r=\max(\sr(A),2)$ and $l\ge 0$.
Then the canonical map
\[
	H_l(\GL_n(A)) \to H_l(\GL_{n+1}(A))
\]
is surjective for $n\ge 2l+r-2$ and bijective for $n\ge 2l+r-1$.
\end{theorem}

Unfortunately, commutative rings really happen to be Tor-unital.
Instead, a recent trend has been to think \textit{Tor-unital pro rings}.
We say that a pro system $\{A_m\}$ of associative rings is \textit{Tor-unital} if the pro system $\{\Tor^{\mbb{Z}\ltimes A_m}_i(\mbb{Z},\mbb{Z})\}_m$ vanish for all $i>0$.
A notable discovery by Morrow \cite{Mo15} is that, for any ideal $A$ of a noetherian commutative ring, the pro ring $\{A^m\}_{m\ge 1}$ is Tor-unital.
Besides, Geisser-Hesselholt \cite{GH06} has generalized Suslin's excision theorem to the pro setting:
If $\{A_m\}$ is a Tor-unital pro ring then the canonical map
\[
	\{K_n(\mbb{Z}\ltimes A_m,A_m)\}_m \xrightarrow{\sim} \{K_n(R_m,A_m)\}_m
\]
is a pro isomorphism for any pro system of unital rings $\{R_m\}$ with a level map $\{A_m\}\to\{R_m\}$ which exhibits each $A_m$ as a two-sided ideal of $R_m$.

Our main theorem is a pro version of Theorem \ref{stab:Su96}.
\begin{theorem}\label{main}
Let $\{A_m\}$ be a commutative Tor-unital pro ring\footnotemark, $r=\max(\sr(A_m),2)$ and $l\ge 0$.
Then the canonical map
\[
	\{H_l(\GL_n(A_m))\}_m \to \{H_l(\GL_{n+1}(A_m))\}_m
\]
is a pro epimorphism for $n\ge 2l+r-2$ and a pro isomorphism for $n\ge 2l+r-1$.
\end{theorem}
\footnotetext{``commutative'' means that each $A_m$ is commutative.
However, this condition may not be essential.
We expect that the theorem is true without the commutativity assumption.}

It follows from Theorem \ref{main} that if $\{A_m\}$ is commutative Tor-unital then the action of $\GL_n(\mbb{Z})$ on $\{H_l(\GL_n(A_m))\}_m$ is pro trivial for $n\ge 2l+r-1$, cf.\ Corollary \ref{cor:protrivial}.
Together with the standard argument this reproves Geisser-Hesselholt's pro excision theorem for commutative Tor-unital rings of finite stable range.

\subsection{Outline}

Suslin used the Malcev theory in the proof of Theorem \ref{stab:Su96}, which works only for $\mbb{Q}$-algebras.
More precisely, he used the Malcev theory to get an acyclicity of the union of triangular spaces, cf.\ \cite[Corollary 5.6]{Su96}.
We prove the pro version of the acyclicity by using a totally different method; it is closer to the methods developed in \cite{Su82,Su95}.

In \S\ref{K1}, we prove the pro stability for $H_1(\GL_n)$, cf.\ Theorem \ref{thm:stabilityK1}.
This essentially follows from Vaser\v{s}te\u\i n's stability for relative $K_1$.
In \S\ref{TorUnital}, we recall some properties of Tor-unital rings, which we need later.
\S\ref{ProAcyclic} is the technical heart of this paper.
In this section, we study triangular spaces and prove a pro acyclicity of the union of triangular spaces, cf.\ Theorem \ref{thm:proacyclic}.
In \S\ref{MainThm}, we complete the proof of Theorem \ref{main}, using the pro acyclicity of triangular spaces.

\subsection{Notation}\label{notation}

\begin{enumerate}[1.]
\item A ring means an associative, not necessarily unital, ring.
\item $\sr(A)$ is the stable range of a ring $A$, i.e.\ the minimum number $r\ge 1$ such that the stable range condition \cite[$(2.2)_n$]{Va69} holds for every $n\ge r$.
\item Let $A$ be a ring and $n\ge 1$.
\begin{enumerate}[(a)]
\item The general linear group $\GL_n(A)$ is the kernel of the canonical map $\GL_n(\mbb{Z}\ltimes A) \to \GL_n(\mbb{Z})$.
\item The elementary subgroup $E_n(A)$ is the subgroup of $\GL_n(A)$ generated by the elementary matrices $e_{ij}(a)$ with $a\in A$ and $1\le i\ne j\le n$.
\end{enumerate}
We regard $\GL_n(A)$ as a subgroup of $\GL_{n+1}(A)$ by sending a matrix $\alpha$ to $\left(\begin{smallmatrix}\alpha & 0 \\0 & 1\end{smallmatrix}\right)$.
We write $\GL(A)=\GL_\infty(A)=\bigcup_n\GL_n(A)$ and $E(A)=E_\infty(A)=\bigcup_nE_n(A)$.
\item A pro ring is a pro system of rings indexed by a filtered poset.
Typically, we denote a pro ring by a bold letter $\mbf{A}=\{A_m\}$ and the structured maps $A_m\to A_n$ by $\iota_{m,n}$ or just by $\iota$.
\item A unital (resp.\ commutative) pro ring is a pro ring which is levelwise unital (resp.\ commutative).
Unless otherwise stated, we use standard operations of rings levelwise for pro rings:
E.g.\ $\GL_n(\mbf{A})=\{\GL_n(A_m)\}_m$, $\Tor^{\mbb{Z}\ltimes\mbf{A}}_*(\mbb{Z},\mbb{Z})=\{\Tor^{\mbb{Z}\ltimes A_m}_*(\mbb{Z},\mbb{Z})\}_m$, etc.
\item A left ideal of a pro ring $\mbf{A}=\{A_m\}_{m\in J}$ is a pro ring $\mbf{B}=\{B_m\}_{m\in J}$ with a level map $\mbf{B}\to\mbf{A}$ which exhibits each $B_m$ as a left ideal of $A_m$.
\end{enumerate}

\subsection{Acknowledgment}

The study of homology pro stability has started when I was working on relative Chern classes with Wataru Kai \cite{IK17}.
I am very grateful to him for many discussions and useful comments.
The majority of this paper was written when I was visiting Moritz Kerz in Universit\"at Regensburg.
I thank him for his kind hospitality and interest in my work.
This work was supported by JSPS KAKENHI Grant Number 16J08843, and by the Program for Leading Graduate Schools, MEXT, Japan.

\section{Pro stability for $K_1$}\label{K1}

\subsection{Vaser\v{s}te\u\i n's stability}

Let $R$ be a unital ring and $A$ a two-sided ideal of $R$.
The normal elementary subgroup $E_n(R,A)$ is the smallest normal subgroup of $E_n(R)$ which contains $E_n(A)$.
We write $E(R,A)=E_\infty(R,A)=\bigcup_nE_n(R,A)$.
By Whitehead's lemma, $E(R,A)$ is a normal subgroup of $\GL(A)$.
We define the relative $K_1$-group $K_1(R,A)$ to be the quotient group $\GL(A)/E(R,A)$.

\begin{theorem}[Vaser\v{s}te\u\i n {\cite{Va69}}]\label{thm:Vaserstein}
The canonical map
\[
	\GL_n(A) \to K_1(R,A)
\]
is surjective for $n\ge\sr(A)$, and the kernel is $E_n(R,A)$ for $n\ge\sr(A)+1$.
\end{theorem}

\subsection{}

Let $R$ be a unital ring and $A$ a two-sided ideal of $R$.
The following lemma generalizes \cite[??]{Ti76} for noncommutative rings.
\begin{lemma}\label{lem:Tits}
For $n\ge 3$, $E_n(R,A^2)\subset [E_n(A),E_n(A)]$.
\end{lemma}
\begin{proof}
Note the standard equality of elementary matrices;
\[
	[e_{ij}(a),e_{kl}(b)] = \begin{cases} 1 &\text{if } j\ne k, i\ne l \\ e_{il}(ab) &\text{if } j=k, i\ne l \\ e_{kj}(-ba) &\text{if }j\ne k, i=l, \end{cases}
\]
which we use throughout the proof.
One immediate consequence is that $E_n(A^2)\subset [E_n(A),E_n(A)]$ for $n\ge 3$.

For $r=(r_1,\dotsc,r_n)\in R^n$ with $r_j=1$, we write
\[
	X_j(r):= \prod_{k\ne j} e_{jk}(r_k) \quad \text{and} \quad X^j(r):=\prod_{k\ne j} e_{kj}(r_k).
\]
Fix $1\le j\le n$.
It is easy to see that every $x\in E_n(R)$ has the form
\[
	x_{2m}(U) := X^j(u_{2m})X_j(u_{2m-1})\dotsm X^j(u_2)X_j(u_1)
\]
for some $m>0$ and $U=(u_1,u_2\dotsc,u_{2m})\in (R^n)^{2m}$.
We also set $x_0(\emptyset):=1$ and 
\[
	x_{2m-1}(V) := X^j(v_{2m-1})X_j(v_{2m-2})\dotsm X_j(v_2)X^j(v_1)
\]
for $m>0$ and $V=(v_1,v_2\dotsc,v_{2m-1})\in (R^n)^{2m-1}$.

Consider the following assertion.
\begin{description}
\item[$(\heartsuit)_N$] For every $U\in (R^n)^N$, $x_N(U) E_n(A^2) x_N(U)^{-1} \subset [E_n(A),E_n(A)]$.
\end{description}
We have seen $(\heartsuit)_0$.
Let $N>0$ and suppose that $(\heartsuit)_l$ holds for $l<N$.
We shall prove $(\heartsuit)_N$ in case $N$ even; the case $N$ odd is proved in the same way.

Let $U=(u_1,\dotsc,u_N)\in (R^n)^N$ and $x:=x_N(U)$.
For $e_{ik}(a)$ with $a\in A^2$, $1\le i,k\le n$ and $k\ne j$, we have $X_j(u_1)e_{ik}(a)X_j(-u_1)\in E_n(A^2)$ and thus by the induction hypothesis $xe_{ik}(a)x^{-1}\in [E_n(A),E_n(A)]$.
For $e_{ij}(a)$ with $a\in A^2$ and $1\le i\ne j\le n$, we have
\[
\begin{split}
	X_j(u_1)e_{ij}(a)X_j(-u_1) &= e_{ji}(u_{1,i}) \Bigl( \prod_{k\ne i,j} e_{ik}(-a u_{1,k}) \cdot e_{ij}(a) \Bigr) e_{ji}(-u_{1,i}) \\
							   &= \prod_{k\ne i,j} e_{jk}(-u_{1,i} a u_{1,k}) e_{ik}(-a u_{1,k}) \cdot e_{ji}(u_{1,i}) e_{ij}(a) e_{ji}(-u_{1,i}).
\end{split}
\]
Hence, it follows from the induction hypothesis that $xE_n(A^2)x^{-1}$ is generated by $y_ie_{ij}(a)y_i^{-1}$, $y_i= X^j(u_N)X_j(u_{N-1})\dotsm X^j(u_2)e_{ji}(u_{1,i})$, with $a\in A^2$ and $1\le i\ne j\le n$ modulo $[E_n(A),E_n(A)]$.

For $U=(u_1,\dotsc,u_N)\in (R^n)^N$ and $1\le p\le N/2$, we set
\begin{gather*}
	y_i^{2p-1}(U): = X^j(u_N)X_j(u_{N-1})\dotsm X^j(u_{2p}) e_{ji}(u_{2p-1,i})\dotsm e_{ij}(u_{2,i})e_{ji}(u_{1,i}) \\
	y_i^{2p}(U): = X^j(u_N)X_j(u_{N-1})\dotsm X_j(u_{2p+1}) e_{ij}(u_{2p,i})\dotsm e_{ij}(u_{2,i})e_{ji}(u_{1,i}).
\end{gather*}
We claim that:
\begin{description}
\item[$(\diamondsuit)_Q$] For $U\in (R^n)^N$, $x_N(U)E_n(A^2)x_N(U)^{-1}$ is generated by $y_i^Q(U)e_{ij}(a) y_i^Q(U)^{-1}$, $a\in A^2$, $1\le i\ne j\le n$ modulo $[E_n(A),E_n(A)]$.
\end{description}
We have seen $(\diamondsuit)_1$.
Let $Q>1$ and suppose that $(\diamondsuit)_l$ holds for $l<Q$.
We prove $(\diamondsuit)_Q$ in case $Q$ even; the case $Q$ odd is proved in the same way.

Let $U=(u_1,\dotsc,u_N)\in (R^n)^N$.
According to $(\diamondsuit)_{Q-1}$, $x_N(U)E_n(A^2)x_N(U)^{-1}$ is generated by $y_i^{Q-1}(U)e_{ij}(a)y_i^{Q-1}(U)^{-1}$, $a\in A^2$, $1\le i\ne j\le n$ modulo $[E_n(A),E_n(A)]$.
We fix $1\le i\ne j\le n$ for a moment.
Now,
\[
	X^j(u_{Q})e_{ji}(u_{Q-1,i}) = e_{ij}(u_{Q,i}) e_{ji}(u_{Q-1,i}) \prod_{k\ne i,j} e_{kj}(u_{Q,k})e_{ki}(u_{Q,k}u_{Q-1,i}).
\]
Hence, by putting $\tilde{y}:=\prod_{k\ne i,j}e_{ki}(u_{2p,k}u_{2p-1,i})$, we have
\[
	y_i^{Q-1}(U) = X^j(u_N)X_j(u_{N-1})\dotsm X_j(u_{Q+1})e_{ij}(u_{Q,i})e_{ji}(u_{Q-1,i})X^j(u_{Q-2}')\dotsm X^j(u_2')X_j(u_1')\tilde{y}
\]
for some $u_1',\dotsc,u_{Q-2}'\in R^n$ with $u'_{q,i}=u_{q,i}$.
For $Q-1\le q\le N$, we set
\[
	u_q':= \begin{cases} u_{q,i}e_i+e_j &\text{if }q=Q-1,Q \\
						 u_q 			&\text{if }q>Q \end{cases}
\]
and $U':=(u_1',\dotsc,u_N')$, so that $y_i^{Q-1}(U)=x_N(U')\tilde{y}$ and $y_i^q(U')=y_i^q(U)$ for $q\ge Q$.
By applying $(\diamondsuit)_{Q-1}$ for $U'$, we see that $x_N(U')E_n(A^2)x_N(U')^{-1}$ is generated by $y_i^Q(U')e_{ij}(a)y_i^Q(U')^{-1}$, $a\in A^2$ modulo $[E_n(A),E_n(A)]$.
Varying $i$, this proves $(\diamondsuit)_Q$ for the given $U\in (R^n)^N$, and thus for all $U\in (R^n)^N$.

According to $(\diamondsuit)_N$, to prove $(\heartsuit)_N$, it suffices to show that $ye_{ij}(ab)y^{-1}\in [E_n(A),E_n(A)]$ for $y= e_{ij}(r_N)e_{ji}(r_{N-1}) \dotsm e_{ij}(r_2)e_{ji}(r_1)$ with $a,b\in A$, $r_1,\dotsc,r_N\in R$ and $1\le i\ne j\le n$.
Observe that we have
\[
\begin{split}
	e_{ij}(r_1)e_{ji}(ab)e_{ij}(-r_1) &= e_{ij}(r_1) [e_{jt}(a),e_{ti}(b)] e_{ij}(-r_1) \\
									  &= [e_{it}(r_1a)e_{jt}(a),e_{tj}(-br_1)e_{ti}(b)]
\end{split}
\]
for $t\ne i,j$.
Now, it is clear that $y'[e_{it}(r_1a)e_{jt}(a),e_{tj}(-br_1)e_{ti}(b)](y')^{-1}\in [E_n(A),E_n(A)]$ for $y'=e_{ij}(r_N) e_{ji}(r_{N-1}) \dotsm e_{ij}(r_2)$, and thus we get $(\heartsuit)_N$.
\end{proof}

\begin{corollary}\label{cor:Tits}
Let $\mbf{R}=\{R_m\}$ be a unital pro ring and $\mbf{A}=\{A_m\}$ a two-sided ideal of $\mbf{R}$.
Suppose that $\mbf{A}/\mbf{A}^2=\{A_m/A_m^2\}=0$.
Then, for $3\le n\le \infty$, the canonical maps
\[
\xymatrix{
	E_n(\mbf{A}) \ar[r]^-\simeq 							 & E_n(\mbf{R},\mbf{A}) \\
	[E_n(\mbf{A}),E_n(\mbf{A})] \ar[u]^\simeq \ar[r]^-\simeq & [E_n(\mbf{R},\mbf{A}),E_n(\mbf{R},\mbf{A})] \ar[u]^\simeq
}
\]
are pro isomorphisms.
\end{corollary}
\begin{proof}
Since all the indicated maps are injections, it suffices to show that the map $[E_n(\mbf{A}),E_n(\mbf{A})]\to E_n(\mbf{R},\mbf{A})$ is a pro epimorphism.
By the assumption $\mbf{A}/\mbf{A}^2=0$, there exists $s\ge m$ for each $m$ such that $\iota_{s,m}(A_s)\subset A_m^2$.
Therefore,
\[
	\iota_{s,m}E_n(R_s,A_s) \subset E_n(R_m,A_m^2) \subset [E_n(R_m,A_m),E_n(R_m,A_m)],
\]
where the last inclusion is by Lemma \ref{lem:Tits}.
This proves that $[E_n(\mbf{A}),E_n(\mbf{A})]\to E_n(\mbf{R},\mbf{A})$ is a pro epimorphism.
\end{proof}

\subsection{Pro excision and pro stability}

Let $\mbf{R}=\{R_m\}$ be a unital pro ring and $\mbf{A}=\{A_m\}$ a two-sided ideal of $\mbf{R}$.
We define $\sr(\mbf{A}):=\max_m(\sr(A_m))$.

\begin{theorem}[Pro exision]\label{thm:excisionK1}
Suppose that $\mbf{A}/\mbf{A}^2=0$.
Then the canonical map
\[
	H_1(\GL(\mbf{A})) \xrightarrow{\sim} K_1(\mbf{R},\mbf{A})
\]
is a pro isomorphism.
\end{theorem}
\begin{proof}
Since $K_1(\mbf{R},\mbf{A})$ is levelwise abelian, we have a levelwise exact sequence
\[
\xymatrix@1{
	H_1(E(\mbf{R},\mbf{A})) \ar[r] & H_1(\GL(\mbf{A})) \ar[r] & K_1(\mbf{R},\mbf{A}) \ar[r] & 0.
}
\]
According to Corollary \ref{cor:Tits}, $H_1(E(\mbf{R},\mbf{A}))=0$, and thus we get $H_1(\GL(\mbf{A})) \simeq K_1(\mbf{R},\mbf{A})$.
\end{proof}

\begin{theorem}[Pro stability]\label{thm:stabilityK1}
Suppose that $\mbf{A}/\mbf{A}^2=0$.
Then the canonical map
\[
	H_1(\GL_n(\mbf{A})) \to H_1(\GL(\mbf{A}))
\]
is a pro epimorphism for $n\ge\sr(\mbf{A})$ and a pro isomorphism for $n\ge\max(3,\sr(\mbf{A})+1)$.
\end{theorem}
\begin{proof}
The composite
\[
	H_1(\GL_n(\mbf{A})) \to H_1(\GL(\mbf{A})) \xrightarrow{\sim} K_1(\mbf{R},\mbf{A})
\]
is a levelwise surjection for $n\ge\sr(\mbf{A})$ by Theorem \ref{thm:Vaserstein}.
Since the last map is a pro isomorphism by Theorem \ref{thm:excisionK1}, the first map is a pro epimorphism for $n\ge\sr(\mbf{A})$.

Consider the commutative diagram
\[
\xymatrix{
	H_1(E_n(\mbf{R},\mbf{A})) \ar[r] \ar[d] & H_1(\GL_n(\mbf{A})) \ar[r] \ar[d] & H_1(\GL_n(\mbf{R},\mbf{A})/E_n(\mbf{R},\mbf{A})) \ar[r] \ar[d] & 0 \\
	H_1(E(\mbf{R},\mbf{A})) \ar[r] 			& H_1(\GL(\mbf{A})) \ar[r] 			& H_1(K_1(\mbf{R},\mbf{A})) \ar[r] & 0
}
\]
with levelwise exact rows.
The left terms are zero for $n\ge 3$ by Corollary \ref{cor:Tits}.
According to Theorem \ref{thm:Vaserstein}, the right vertical map is a levelwise bijection for $n\ge\sr(\mbf{A})+1$.
Hence, the middle term is a pro isomorphism for $n\ge\max(3,\sr(\mbf{A})+1)$.
\end{proof}

\begin{theorem}\label{thm2:stabilityK1}
Set $\bar{E}_n(\mbf{A}):=\GL_n(\mbf{A})\cap E(\mbf{A})$.
Suppose that $\mbf{A}/\mbf{A}^2=0$.
Then the canonical map
\[
	E_n(\mbf{A}) \to \bar{E}_n(\mbf{A})
\]
is a pro isomorphism for $n\ge\max(3,\sr(\mbf{A})+1)$.
\end{theorem}
\begin{proof}
Let $\bar{E}_n(\mbf{R},\mbf{A}):=\GL_n(\mbf{A})\cap E(\mbf{R},\mbf{A})$.
According to Theorem \ref{thm:Vaserstein}, the canonical map $E_n(\mbf{R},\mbf{A})\to\bar{E}_n(\mbf{R},\mbf{A})$ is a levelwise bijection for $n\ge \sr(\mbf{A})+1$.
Hence, the theorem follows from Corollary \ref{cor:Tits}.
\end{proof}

\section{Tor-unital pro rings}\label{TorUnital}

The treatment of this section closely follows Suslin \cite{Su95} and Geisser-Hesselholt \cite{GH06}.

\subsection{Definitions}

\begin{definition}\label{def:TorUnital}
A pro ring $\mbf{A}=\{A_m\}$ is \textit{Tor-unital} if 
\[
	\Tor^{\mbb{Z}\ltimes\mbf{A}}_i(\mbb{Z},\mbb{Z})=\{\Tor^{\mbb{Z}\ltimes A_m}_i(\mbb{Z},\mbb{Z})\}_m = 0
\]
as pro abelian groups for all $i>0$.
\end{definition}

\begin{example}
\leavevmode
\begin{enumerate}[(i)]
\item A unital pro ring, i.e.\ a pro system of unital rings, is Tor-unital.
\item (Morrow \cite{Mo15}) Let $A$ be an ideal of a noetherian commutative ring, then $\{A^m\}_{m\ge 1}$ is Tor-unital.
\end{enumerate}
\end{example}

\begin{definition}\label{def:special}
Let $\mbf{A}=\{A_m\}_{m\in J}$ be a pro ring.
\begin{enumerate}[(i)]
\item A \textit{left $\mbf{A}$-module} is a pro abelian group $\mbf{M}=\{M_m\}_{m\in J}$ with a level map $\mbf{A}\times\mbf{M}\to\mbf{M}$ which exhibits each $M_m$ as a left $A_m$-module.
A \textit{morphism between left $\mbf{A}$-modules $\mbf{M}=\{M_m\}$ and $\mbf{N}=\{N_m\}$} is a level map $f\colon \mbf{M}\to\mbf{N}$ such that each $f_m\colon M_m\to N_m$ is a morphism of left $A_m$-modules.
\item A left $\mbf{A}$-module $\mbf{P}$ is \textit{pseudo-free} if there is an isomorphism of left $\mbf{A}$-modules $\mbf{A}\otimes\mbf{L}\xrightarrow{\simeq}\mbf{P}$ for some pro system $\mbf{L}$ of free abelian groups.
We call such an $\mbf{L}$ a \textit{free basis of $\mbf{P}$}.
\item A morphism $f\colon\mbf{P}\to\mbf{M}$ of left $\mbf{A}$-modules is \textit{special} if $\mbf{P}$ is pseudo-free with a free basis $\mbf{L}$ and $f$ is induced from a level morphism of pro abelian groups $\mbf{L}\to\mbf{M}$.
\end{enumerate}
\end{definition}

\subsection{Pro resolution}\label{proresol}

\begin{proposition}[Suslin \cite{Su95}, Geisser-Hesselholt \cite{GH06}]\label{prop:TorUnital}
Let $\mbf{A}=\{A_m\}$ be a Tor-unital pro ring.
Suppose we are given an augmented complex
\[
	\dotso \to \mbf{C}_1 \to \mbf{C}_0 \xrightarrow{\epsilon} \mbf{C}_{-1}
\]
of left $\mbf{A}$-modules such that:
\footnote{We thank Takeshi Saito for pointing out an unnecessary condition, the augmentation $\epsilon$ is special, which was in the first draft and in \cite{Su95,GH06} too.}
\begin{enumerate}[(i)]
\item Each $\mbf{C}_k$ with $k\ge -1$ is pseudo-free.
\item The homology $H_k(C_{\bullet,m})$ is annihilated by $A_m$ for every $m$ and $k\ge -1$.
\end{enumerate}
Then
\[
	H_k(\mbf{C}_\bullet) = \{H_k(C_{\bullet,m})\}_m = 0
\]
for all $k\ge -1$.
\end{proposition}

In fact, a finer result holds.
\begin{proposition}\label{prop2:TorUnital}
Let $\mbf{A}=\{A_m\}_{m\in J}$ be a Tor-unital pro ring and $k\ge -1$.
Then there exists $s(m)\ge m$ for each $m\in J$ such that the map
\[
	\iota_{s(m),m}\colon H_k(C_{\bullet,s(m)}) \to H_k(C_{\bullet,m})
\]
is zero for all augmented complexes of left $\mbf{A}$-modules which satisfy the conditions (i) (ii).
\end{proposition}

\begin{proof}
Let $\mbf{C}$ be a pseudo-free left $\mbf{A}$-module with a free basis $\mbf{L}$.
Then we have levelwise isomorphisms
\[
\begin{split}
	\Tor^{\mbb{Z}\ltimes\mbf{A}}_q(\mbb{Z},\mbf{C}) &\simeq \Tor^{\mbb{Z}\ltimes\mbf{A}}_q(\mbb{Z},\mbf{A}\otimes\mbf{L}) \\
											 		&\simeq \Tor^{\mbb{Z}\ltimes\mbf{A}}_q(\mbb{Z},\mbf{A})\otimes\mbf{L} \\
											 		&\simeq \Tor^{\mbb{Z}\ltimes\mbf{A}}_{q+1}(\mbb{Z},\mbb{Z})\otimes\mbf{L}.
\end{split}
\]
Since $\mbf{A}$ is Tor-unital, we see that 
\[
	\Tor^{\mbb{Z}\ltimes\mbf{A}}_q(\mbb{Z},\mbf{C}) = 0
\]
for every $q\ge 0$.

Let $\mbf{Z}_k$ and $\mbf{B}_{k-1}$ are the kernel and the image of $\mbf{C}_k\to \mbf{C}_{k-1}$ respectively.
By the assumption (ii), we have a levelwise inclusion $\mbf{A}\mbf{C}_{-1}\subset \mbf{B}_{-1}$, and thus there is a levelwise surjection $\mbf{C}_{-1}/\mbf{A}\mbf{C}_{-1}\twoheadrightarrow H_{-1}(\mbf{C}_\bullet)$.
Since $\mbf{C}_{-1}$ is pseudo-free, $\mbf{C}_{-1}/\mbf{A}\mbf{C}_{-1} = \Tor^{\mbb{Z}\ltimes\mbf{A}}_0(\mbb{Z},\mbf{C}_{-1})=0$.
Therefore, $H_{-1}(\mbf{C}_\bullet)=0$.

Let $k\ge 0$ and suppose that $H_l(\mbf{C}_\bullet) = 0$ for $l<k$.
Consider the levelwise spectral sequence
\[
	\mbf{E}^1_{pq} = \begin{cases}
						\Tor^{\mbb{Z}\ltimes\mbf{A}}_q(\mbb{Z},\mbf{C}_p) &\text{if } 0\le p \le k \\
						\Tor^{\mbb{Z}\ltimes\mbf{A}}_q(\mbb{Z},\mbf{Z}_k) &\text{if } p=k+1 \\
						0 &\text{otherwise}, \end{cases}
\]
which arises from the brutal truncation of the complex
\[
	\mbf{Z}_k \to \mbf{C}_k \to \mbf{C}_{k-1} \to \dotso \to \mbf{C}_0.
\]
By the induction hypothesis, the complex is pro quasi-isomorphic to $\mbf{C}_{-1}$ and thus
\[
	\mbf{E}^\infty_q \simeq \Tor^{\mbb{Z}\ltimes\mbf{A}}_q(\mbb{Z},\mbf{C}_{-1}) = 0
\]
for $q\ge 0$.
Since $\mbf{C}_p$ is pseudo-free, we also have $\mbf{E}^1_{pq}=0$ for $0\le p\le k$.
Hence, 
\[
	\mbf{Z}_k/\mbf{A}\mbf{Z}_k = \Tor^{\mbb{Z}\ltimes\mbf{A}}_0(\mbb{Z},\mbf{Z}_k) = \mbf{E}^\infty_{k+1} = 0.
\]
On the other hand, by the assumption (ii), we have $\mbf{A}\mbf{Z}_k\subset \mbf{B}_k$.
Therefore, $H_k(\mbf{C}_\bullet)= 0$.

A finer variant is also clear from this proof.
\end{proof}

\begin{lemma}\label{lem:proresol}
Let $\mbf{A}$ be a pro ring and $\mbf{P}$ a pseudo-free left $\mbf{A}$-module.
Then there exists an augmented complex $\mbf{P}_\bullet$ of left $\mbf{A}$-modules with $\mbf{P}_{-1}=\mbf{P}$ which satisfies the conditions (i) (ii) and
\begin{enumerate}[(iii)]
\item The augmentation $\epsilon\colon\mbf{P}_0 \to \mbf{P}_{-1}$ is special
\end{enumerate}
\end{lemma}
We call $\mbf{P}_\bullet$ a \textit{pro resolution} of $\mbf{P}$.
\begin{proof}
Write $\mbf{P}=\{P_m\}$ and let $\mbb{Z}[\mbf{P}]=\{\mbb{Z}[P_m]\}$ be the pro system of the free abelian groups generated by the sets $P_m$.
Then $\mbf{P}_0 := \mbf{A} \otimes \mbb{Z}[\mbf{P}]$ is a pseudo-free $\mbf{A}$-module and the canonical map $\mbb{Z}[\mbf{P}]\to\mbf{P}$ induces a special morphism $\epsilon\colon\mbf{P}_0\to\mbf{P}$.

Let $\mbf{R}=\{R_m\}$ be the kernel of $\epsilon$, and $\mbb{Z}[\mbf{R}]=\{\mbb{Z}[R_m]\}$ the pro system of the free abelian group generated by $R_m$.
Then $\mbf{P}_1 := \mbf{A}\otimes \mbb{Z}[\mbf{R}]$ is a pseudo-free $\mbf{A}$-module.
Repeating this procedure, we obtain an augmented complex $\mbf{P}_\bullet$ with $\mbf{P}_{-1}=\mbf{P}$ which satisfies the desired conditions.
\end{proof}

\section{Pro acyclicity of triangular spaces}\label{ProAcyclic}

The goal of this section is to prove Theorem \ref{thm:proacyclic}.

\subsection{Preliminaries on homology}\label{grouphomology}

\subsubsection{}

For a simplicial set $X$, we denote by $C_*(X)$ the complex freely generated by $X_*$ with the differential being the alternating sum of the faces.
We write $H_n(X)=H_n(C_*(X))$.
Also, we write $\tilde{H}_n(X)$ for the reduced homology.

Let $G$ be a group.
We write $EG$ for the simplicial set whose degree $n$-part is $G^{\times(n+1)}$ and whose $i$-th face (resp.\ the $i$-th degeneracy) omits the $i$-th entry (resp.\ repeats the $i$-th entry).
We give a right $G$-action on $EG$ by $(g_0,\dotsc,g_n)\cdot g := (g_0g,\dotsc,g_ng)$.
The classifying space $BG$ is defined to be $EG/G$.

\subsubsection{}

By a pro object or pro system, we mean a pro object whose index category is a left filtered small category.

\begin{lemma}\label{lem:proC_*}
Let $f\colon X\to Y$ be a morphism between pro systems of pointed simplicial sets.
Suppose that $f$ induces pro isomorphisms
\[
	\pi_n(X) \xrightarrow{\sim} \pi_n(Y)
\]
for all $n\ge 0$.
Then $f$ induces pro isomorphisms
\[
	H_n(X) \xrightarrow{\sim} H_n(Y)
\]
for all $n\ge 0$.
\end{lemma}
\begin{proof}
Since $\mbb{Z}\pi_0(X)\simeq H_0(X)$, the assertion is clear for $n=0$.
Hence, by taking the connected components, we may assume that $X$ and $Y$ are connected.

Then, according to \cite{Is01}, for each $n\ge 1$,
the induced map $P_n(X) \to P_n(Y)$ on the $n$-th Postnikov sections is a strict weak equivalence, i.e.\ isomorphic to a levelwise weak equivalence.
Hence, the induced map $C_*(P_n(X))\to C_*(P_n(Y))$ is isomorphic to a levelwise quasi-isomorphism.

On the other hand, by Hurewicz theorem and Serre spectral sequence, we have levelwise isomorphisms $H_k(X) \simeq H_k(P_n(X))$ for $k\le n$.
Now, in the commutative diagram
\[
\xymatrix{
	H_k(X) \ar[r] \ar[d]^\simeq 	& H_k(Y) \ar[d]^\simeq \\
	H_k(P_n(X)) \ar[r]^\simeq 	& H_k(P_n(Y)),
}
\]
the vertical maps and the bottom map are pro isomorphisms for $k\le n$, and so is the top map.
Since $n$ is arbitrary, $H_k(X)\xrightarrow{\simeq}H_k(Y)$ is a pro isomorphism for every $k\ge 0$.
\end{proof}

For a simplicial group $G$, we consider the bi-simplicial set $BG$ constructed degreewise.
For a bi-simplicial set $X$, we denote by $C_*(X)$ the double-complex freely generated by $X_*$ with the differential being the alternating sum of the faces.

\begin{corollary}\label{cor:proC_*}
Let $f\colon P \to Q$ be a morphism between pro systems of simplicial abelian groups.
Suppose that $f$ induces pro isomorphisms
\[
	\pi_n(P) \xrightarrow{\sim} \pi_n(Q)
\]
for all $n\ge 0$.
Then $f$ induces pro isomorphisms
\[
	H_n(\Tot C_*(BP)) \xrightarrow{\sim} H_n(\Tot C_*(BQ))
\]
for all $n\ge 0$.
\end{corollary}
\begin{proof}
Now, the morphism $B_k P \to B_k Q$ induces pro isomorphisms $\pi_n(B_kP)\to \pi_n(B_kQ)$ for all $n\ge 0$.
Hence, by Lemma \ref{lem:proC_*}, the induced maps
\[
	H_n(C_*(B_kP))\to H_n(C_*(B_kQ))
\]
are pro isomorphisms for all $n\ge 0$.
By the standard spectral sequence, we obtain the corollary.
\end{proof}

\subsubsection{}

Let us quote a lemma from \cite[\S2]{Su95}.
\begin{lemma}\label{lem:semidirect}
Let $G$ be a group and $H$ a group with a left $G$-action.
Then there exists a natural quasi-isomorphism
\[
	C_*(B(G\ltimes H)) \simeq C_*(EG)\otimes_G C_*(BH).
\]
\end{lemma}

Let $G=\{G_m\}$ be a pro group ($=$ a pro system of groups).
A \textit{left $G$-module $M$} is a pro abelian group $M=\{M_m\}$ with a level map $G\times M\to M$ which exhibits each $M_m$ as a left $G_m$-module.
A \textit{morphism between left $G$-modules $M=\{M_m\}$ and $N=\{N_m\}$} is a level map $f\colon M\to N$ such that each $f_m\colon M_m\to N_m$ is a morphism of left $G_m$-modules.
These form the category of left $G$-modules, and we consider simplicial objects in this category; simplicial left $G$-modules and morphisms between them.

\begin{corollary}\label{cor:semidirect}
Let $G$ be a pro group.
Let $P$ and $Q$ be simplicial left $G$-modules and $f\colon P \to Q$ a morphism between them.
Suppose that $f$ induces pro isomorphisms
\[
	\pi_n(P) \xrightarrow{\sim} \pi_n(Q)
\]
for all $n\ge 0$.
Then $f$ induces pro isomorphisms
\[
	H_n(\Tot C_*(B(G\ltimes P))) \xrightarrow{\sim} H_n(\Tot C_*(B(G\ltimes Q)))
\]
for all $n\ge 0$, where the semi-direct products are taken levelwise and degreewise.
\end{corollary}
\begin{proof}
This follows from Corollary \ref{cor:proC_*} and Lemma \ref{lem:semidirect}.
\end{proof}

\subsection{The key lemma}\label{keylem}

\subsubsection{Triangular spaces}\label{triangular}
Let $A$ be a ring and $P$ a left $A$-module.
Let $\sigma$ be a finite poset.
We define a group $T^\sigma(A,P)$ by
\[
	T^\sigma(A,P) := \prod_{i<_\sigma j,\,j\notin\max\sigma} A_{(i,j)} \times \prod_{i<_\sigma j, \,j\in\max\sigma} P_{(i,j)},
\]
where $A_{(i,j)}$ and $P_{(i,j)}$ are just the copies of $A$ and $P$ respectively.
For $\alpha\in T^\sigma(A,P)$, we denote its $(i,j)$-th component by $\alpha_{i,j}$; thus $\alpha_{i,j}\in A$ if $j\notin\max\sigma$, and $\alpha_{i,j}\in P$ if $j\in\max\sigma$.
For $\alpha,\beta\in T^\sigma(A,P)$, we define the composition $\alpha\cdot\beta$ by
\[
	(\alpha\cdot\beta)_{i,j} = \alpha_{i,j} + \beta_{i,j} + \sum_{i<_\sigma k <_\sigma j} \alpha_{i,k} \beta_{k,j}
\]
for $i<_\sigma j$.
We shorten $T^\sigma(A) := T^\sigma(A,A)$.

Set $\sigma_0:=\sigma\setminus\max\sigma$.
Then we have an identification
\[
	T^\sigma(A,P) = T^{\sigma_0}(A) \ltimes M_{\sigma_0,\max\sigma}(P),
\]
and canonical inclusion and projection
\[
	T^{\sigma_0}(A) \hookrightarrow T^\sigma(A,P) \twoheadrightarrow T^{\sigma_0}(A).
\]

Let $\theta\colon\sigma\to\tau$ be a morphism of finite posets.
Then it induces a morphism of groups 
\[
	T^\theta\colon T^\sigma(A) \to T^\tau(A).
\]
If $\theta^{-1}(\max\tau)=\max(\sigma)$, then it also induces a morphism $T^\sigma(A,P) \to T^\tau(A,P)$ for any left $A$-module $P$, which we also denote by $T^\theta$.

Let $f\colon P\to Q$ be a morphism of $A$-modules.
Then it induces a morphism of groups
\[
	T^f\colon T^\sigma(A,P) \to T^\sigma(A,Q)
\]
If $\theta\colon\sigma\to\tau$ satisfies $\theta^{-1}(\max\tau)=\max(\sigma)$, then we define 
\[
	T^{f,\theta}\colon T^\sigma(A,P)\to T^\tau(A,Q)
\]
to be the composite $T^f\circ T^\theta = T^\theta\circ T^f$.

\subsubsection{}
For a finite poset $\sigma$ and $p\ge 0$, let $[p]$ be the poset $0<1<2<\dotsb<p$ and endow $\sigma\times [p]$ with the lexicographical order.
We define
\[
	\sigma\bigstar [p] := \sigma\times [p] \setminus \max\sigma\times\{1,\dotsc,p\}.
\]
We denote by $\phi$ (resp.\ $\varphi$) the embedding $\sigma \to \sigma\times [p]$ (resp.\ $\sigma \to \sigma\bigstar [p]$), $a\mapsto (a,0)$.
Note that $\varphi^{-1}(\max(\sigma\bigstar [p]))=\max\sigma$ and that $(\sigma\bigstar [p])_0 = \sigma_0\times [p]$.

The following lemma is a variant of Lemma 7.4 in \cite{Su82}.
\begin{lemma}\label{lem:key}
Let $\{A_m\}_{m\in\Xi}$ be a commutative Tor-unital pro ring and $l\ge 0$.
Then there exist $p_l\ge 0$ and $s_l(m)\ge m$ for each $m\in\Xi$ such that:
\begin{enumerate}[(i)]
\item For all finite posets $\sigma$ and all pseudo-free $\{A_m\}$-modules $\{P_m\}$, the map
\[
	\iota_{s_l(m),m}H_l(T^\varphi) \colon \tilde{H}_l(T^\sigma(A_{s_l(m)},P_{s_l(m)}) \to \tilde{H}_l(T^{\sigma\bigstar [p_l]}(A_m,P_m))
\]
is equal to zero.
\item For all finite posets $\sigma$ and all special morphisms $f\colon \{P_m\}\to \{Q_m\}$ between pseudo-free $\{A_m\}$-modules, the map
\[
	\iota_{s_l(m),m}H_l(T^{f,\varphi}) \colon \tilde{H}_l(T^\sigma(A_{s_l(m)},P_{s_l(m)}) \to \tilde{H}_l(T^{\sigma\bigstar [p_l]}(A_m,Q_m))
\]
is equal to zero.
\end{enumerate}
\end{lemma}
\begin{proof}
We prove the lemma by induction on $l\ge 0$.
The case $l=0$ is clear, here we can take $p_0=0$ and $s_0(m)=m$.
Let $L>0$ and suppose that we have constructed $p_0\le p_1\le \dotsb \le p_{L-1}$ and $s_0(m)\le s_1(m)\le \dotsb \le s_{L-1}(m)$ which satisfy the conditions of the lemma.

Set $q:=p_{L-1}+1$ and $t(m):=s_{L-1}(m)$.
First, we prove the following.
\begin{sublemma}\label{sublem:key}
For all finite posets $\sigma$ and all special morphisms $f\colon \{P_m\}\to \{Q_m\}$ between pseudo-free $\{A_m\}$-modules, the diagram
\[
\xymatrix@C+2pc{
	H_L(T^\sigma(A_{t(m)},P_{t(m)}) \ar[r]^-{\iota_{t(m),m}H_l(T^{f,\varphi})} \ar@{->>}[d] & H_L(T^{\sigma\bigstar [q]}(A_m,Q_m)) \\
	H_L(T^{\sigma_0}(A_{t(m)}) \ar[r]^-{\iota_{t(m),m}H_l(T^\phi)}							& H_L(T^{\sigma_0\times [q]}(A_m)) \ar@{^{(}->}[u]
}
\]
commutes, where the vertical maps are the canonical projection and inclusion.
\end{sublemma}
\begin{proof}
Let $f\colon\{P_m\}\to\{Q_m\}$ be a special morphism between pseudo-free $\{A_m\}$-modules and $\{L_m\}$ a free basis of $\{P_m\}$ such that $f$ is induced from a map $\{L_m\}\to\{Q_m\}$.
Note that we have an equality $\{\varinjlim_i L_m^{(i)}\} = \{L_m\}$, where $\{L_m^{(i)}\}$ is a sub-system of $\{L_m\}$ such that each $L_m^{(i)}$ is finitely generated and the limit runs around all such systems.
Hence, we have
\[
	 \varinjlim_i C_*(BM_{n,k}(A_m\otimes L_m^{(i)})) \simeq C_*(BM_{n,k}(P_m))
\]
for every $m$ and $n,k\ge 1$.
It follows that
\[
	C_*(BT^\sigma(A_m,P_m)) \simeq \varinjlim_i C_*(BT^\sigma(A_m,A_m\otimes L_m^{(i)}))
\]
and
\[
	H_*(T^\sigma(A_m,P_m)) \simeq \varinjlim_i H_*(T^\sigma(A_m,A_m\otimes L_m^{(i)})).
\]
Consequently, to show the sublemma, we may assume that $\{P_m\}=\{A_m\otimes_{\mbb{Z}} L_m\}$ with $L_m$ a free abelian group of finite rank.
We may also assume that $\{Q_m\} = \{A_m\otimes_{\mbb{Z}} M_m\}$ with $M_m$ a free abelian group of finite rank.

Fix $m\in\Xi$.
Let $e_1,\dotsc,e_I$ be a basis of $L_{t(m)}$ and $f_1,\dotsc,f_J$ a basis of $M_m$.
Since $f$ is special, the map $\iota_{t(m),m}f \colon P_{t(m)}\to Q_m$ is induced by a map $\alpha\colon L_{t(m)}\to Q_m$, which sends $e_i$ to $\sum_j \alpha_{i,j} f_j$ with $\alpha_{i,j}\in A_m$.
We may also denote $\iota_{t(m),m}f$ by $\alpha$.

If $\alpha=0$, then the diagram
\[
\xymatrix@C+2pc{
	H_L(T^\sigma(A_{t(m)},P_{t(m)}) \ar[r]^{\iota_{t(m),m}H_l(T^{f,\varphi})} \ar@{->>}[d] 	& H_L(T^{\sigma\bigstar [q]}(A_m,Q_m)) \\
	H_L(T^{\sigma_0}(A_{t(m)}) \ar[r]^{\iota_{t(m),m}H_l(T^\phi)}						& H_L(T^{\sigma_0\times [q]}(A_m)) \ar@{^{(}->}[u]
}
\]
commutes, and thus the sublemma holds in this case.
Let $(u,v)\in [1,I]\times[1,J]$ and suppose that the sublemma holds if $\alpha_{i,j} = 0$ for $(i,j) \ge (u,v)$ with respect to the lexicographical order.
We prove the sublemma in case $\alpha_{i,j} = 0$ for $(i,j)>(u,v)$.
We define a morphism $\beta\colon P_{t(m)}\to Q_m$ by sending $e_i$ to $\delta_{i,u}f_v$.

We define an embedding $\psi \colon \sigma \to \sigma\bigstar [q]$ by
\[
	\psi(x) = \begin{cases} 	(x,0) &\text{if }x\in\max\sigma \\
								(x,q) &\text{if }x\notin\max\sigma. \end{cases}
\]
Then the image $\tau$ of $\psi$ intersects with $\sigma\bigstar [q-1]\subset \sigma\bigstar [q]$ only at $\max\sigma\times\{0\}$, and thus the composite
\[
\xymatrix@1{
	T^\sigma \ar[r]^-{\text{diag}} & T^\sigma\times T^\sigma \ar[r]^-{T^\varphi\times T^\psi} & T^{\sigma\bigstar [q-1]} \times T^\tau \ar[r]^-{\text{prod}} & T^{\sigma\bigstar [q]}
}
\]
is a group homomorphism.
Now, we have a group morphism
\[
\xymatrix@R-2.5pc{
			& T^{\sigma\bigstar [q-1]}(A_m,Q_m) \ar@{}[dd]|\times \ar[rd] & \\
	T^\sigma(A_{t(m)},P_{t(m)}) \ar[ru]^-{T^{\alpha,\varphi}} \ar[rd]_-{T^{\beta,\psi}} & & T^{\sigma\bigstar [q]}(A_m,Q_m), \\
			& T^\tau(A_m,Q_m) \ar[ur] &
}
\]
which we denote by $T^{\alpha,\varphi}\cdot T^{\beta,\psi}$.
Since $q-1=p_{L-1}$ and $t(m)=s_{L-1}(m)$, by the induction hypothesis and by the K\"unneth formula, we obtain
\begin{equation}\label{eq1:key}
	H_L(T^{\alpha,\varphi}\cdot T^{\beta,\psi}) = H_L(T^{\alpha,\varphi}) + H_L(T^{\beta,\psi}).
\end{equation}

Set
\[
	\omega := \prod_{x\in \sigma_0} e_{\varphi(x),\psi(x)}(\alpha_{u,v}) \in T^{\sigma\bigstar [q]}(A_m).
\]
We define $(\alpha'_{i,j})\in M_{I,J}(A^m)$ by $\alpha'_{i,j}=\alpha_{i,j}$ unless $(i,j)=(u,v)$ and $\alpha'_{u,v}=0$,
which induces a map $\alpha'\colon Q_{t(m)}\to P_m$ by sending $e_i$ to $\sum_j \alpha_{i,j}'f_j$.
\begin{claim}
We have an equality \footnotemark
\footnotetext{Here is the only place we need the commutativity of pro rings}
\[
	\Ad(\omega)\circ(T^{\alpha',\varphi}\cdot T^{\beta,\psi}) = T^{\alpha,\varphi}\cdot T^{\beta,\psi}.
\]
\end{claim}
We calculate the $(k,l)$-entry of (\ref{eq2:key}) at $U\in T^\sigma(A_{t(m)},P_{t(m)})$.
It suffices to do this for;
\begin{enumerate}[(1)]
\item $(k,l)=(\varphi(x),\varphi(y))$ with $x\in\sigma_0$ and $y\in\sigma$.
\item $(k,l)=(\varphi(x),\psi(y))$ with $x\in\sigma_0$ and $y\in\sigma_0$.
\item $(k,l)=(\psi(x),\psi(y))$ with $x\in\sigma_0$ and $y\in\sigma_0$.
\item $(k,l)=(\psi(x),\varphi(y))$ with $x\in\sigma_0$ and $y\in\sigma$.
\end{enumerate}
Case (1):
\[
\begin{split}
	&\bigl(\Ad(\omega)\circ(T^{\alpha',\varphi} \cdot T^{\beta,\psi})(U)\bigr)_{\varphi(x),\varphi(y)} \\
	&= \bigl((T^{\alpha',\varphi}\cdot T^{\beta,\psi})(U)\bigr)_{\varphi(x),\varphi(y)} + \alpha_{u,v}\cdot \bigl((T^{\alpha',\varphi}\cdot T^{\beta,\psi})(U)\bigr)_{\psi(x),\varphi(y)}  \\
	&= T^{\alpha',\varphi}(U)_{\varphi(x),\varphi(y)} + \alpha_{u,v} \cdot T^{\beta,\psi}(U)_{\psi(x),\varphi(y)} \\ 
	&= \begin{cases} U_{x,y} &\text{if }y\in\sigma_0 \\ \alpha'(U_{x,y})+\beta(U_{x,y})\alpha_{u,v} = \alpha(U_{x,y}) &\text{if }y\in\max\sigma \end{cases} \\
	&= \bigl((T^{\alpha,\varphi}\cdot T^{\beta,\psi})(U)\bigr)_{\varphi(x),\varphi(y)}.
\end{split}
\]
Case (2):
\[
\begin{split}
	&\bigl((\Ad(\omega)\circ(T^{\alpha',\varphi}\cdot T^{\beta,\psi})(U)\bigr)_{\varphi(x),\psi(y)} \\
	&= \alpha_{u,v}\cdot\bigl((T^{\alpha',\varphi}\cdot T^{\beta,\psi})(U)\bigr)_{\varphi(x),\varphi(y)} - \bigl((T^{\alpha',\varphi}\cdot T^{\beta,\psi})(U)\bigr)_{\psi(x),\psi(y)}\cdot\alpha_{u,v} \\
	&= \alpha_{u,v}U_{x,y} - U_{x,y}\alpha_{u,v} \\
	&= 0 \\
	&= \bigl((T^{\alpha,\varphi}\cdot T^{\beta,\psi})(U)\bigr)_{\varphi(x),\psi(y)}.
\end{split}
\]
Case (3):
\[
\begin{split}
	&\bigl(\Ad(\omega)\circ(T^{\alpha',\varphi}\cdot T^{\beta,\psi})(U)\bigr)_{\psi(x),\psi(y)} \\
	&= \bigl((T^{\alpha',\varphi}\cdot T^{\beta,\psi})(U)\bigr)_{\psi(x),\psi(y)} - \bigl((T^{\alpha',\varphi}\cdot T^{\beta,\psi})(U)\bigr)_{\psi(x),\phi(y)}\cdot\alpha_{u,v} \\
	&= \bigl((T^{\alpha,\varphi}\cdot T^{\beta,\psi})(U)\bigr)_{\psi(x),\psi(y)}.
\end{split}
\]
Case (4):
\[
\begin{split}
	\bigl((\Ad(\omega)\circ(T^{\alpha',\varphi}\cdot T^{\beta,\psi})(U)\bigr)_{\psi(x),\varphi(y)} 	&= \bigl((T^{\alpha',\varphi}\cdot T^{\beta,\psi}(U)\bigr)_{\psi(x),\varphi(y)} \\
																					&= \bigl((T^{\alpha',\varphi}\cdot T^{\beta,\psi}(U)\bigr)_{\psi(x),\varphi(y)}.
\end{split}
\]
Consequently, we obtain the equality in the claim.

Again, by the induction hypothesis and by the K\"unneth formula, we obtain
\begin{equation}\label{eq2:key}
\begin{split}
	H_L(T^{\alpha,\varphi}\cdot T^{\beta,\psi}) 	&= H_L(T^{\alpha',\varphi}\cdot T^{\beta,\psi}) \\
											&= H_L(T^{\alpha',\varphi}) + H_L(T^{\beta,\psi}).
\end{split}
\end{equation}
It follows from (\ref{eq1:key}, \ref{eq2:key}) that
\[
	H_L(T^{\alpha,\varphi}) = H_L(T^{\alpha',\varphi}).
\]
Therefore, by induction, we get the sublemma.
\end{proof}

We return to the proof of Lemma \ref{lem:key}.
We prove $(i)_{l=L}$.
Let $\{P_m\}$ be a pseudo-free $\{A_m\}$-module.
Let $\{P_m[-]\}$ be a pro resolution of $\{P_m\}$, cf.\ Lemma \ref{lem:proresol}.
Then, by Proposition \ref{prop:TorUnital} and Corollary \ref{cor:semidirect}, $\{P_m[-_{\ge 0}]\}\to \{P_m\}$ induces a pro isomorphism
\[
	\Theta \colon \{H_L(T^\sigma(A_m,P_m[-_{\ge 0}]))\}_m \xrightarrow{\sim} \{H_L(T^\sigma(A_m,P_m))\}_m.
\]
In fact, for each $m\in\Xi$ there exists $r(m)\ge m$, which does not depend on $\{P_m\}$, $\{P_m[-]\}$ and $\sigma$, such that the maps $\iota_{r(m),m}\colon \ker\Theta_{r(m)}\to \ker\Theta_{m}$ and $\iota_{r(m),m}\colon \coker\Theta_{r(m)}\to \coker\Theta_{m}$ are equal to zero, cf.\ Proposition \ref{prop2:TorUnital}.

We set
\begin{gather*}
	p:=p_L:= \Bigl(\prod_{l=1}^{L-1}(p_l+1)\Bigr) (q+1)-1, \\
	s(m):=s_L(m):=r(s_1(\dotsb(s_{L-1}(t(m)))\dotsb)).
\end{gather*}
We claim that $(i)_{l=L}$ holds with these definitions.
We prove it by induction on $n:=\#\sigma\ge 1$.
The case $n=1$ is clear, and so let $n>1$.

By Lemma \ref{lem:semidirect}, we have
\[
	C_*(BT^\sigma(A_m,P_m[-_{\ge 0}]) = C_*(ET^{\sigma_0}(A_m))\otimes_{T^{\sigma_0}(A_m)}C_*(BM_{\sigma_0,\max\sigma}(P_m[-_{\ge 0}]))
\]
and thus we have a first quadrant spectral sequence
\[
	(E^1_{s,t})^\sigma_m = H_t(T^\sigma(A_m,P_m[s])) \Rightarrow H_{s+t}(T^\sigma(A_m,P_m[-_{\ge 0}])).
\]
It is clear that $(E^2_{s,0})^\sigma_m=0$ for $s>0$.
Hence, the spectral sequence induces a filtration
\[
	0=F_{-1,m}^\sigma \subset F_{0,m}^\sigma \subset \dotsb \subset F_{L-1,m}^\sigma=H_L(T^\sigma(A_m,P_m[-_{\ge 0}]))
\]
with $F_{i,m}^\sigma/F_{i-1,m}^\sigma$ a subquotient of $H_{L-i}(T^\sigma(A_m,P_m[i]))$.

Note that the map $\varphi\colon \sigma\to \sigma\bigstar [p]$ induces a morphism of spectral sequences 
\[
\xymatrix@R-1pc@C-1pc{
	(E^1_{s,t})^\sigma_m = H_t(T^\sigma(A_m,P_m[s])) \ar@{=>}[r] \ar[d] 					 & H_{s+t}(T^\sigma(A_m,P_m[-_{\ge 0}])) \ar[d] \\
	(E^1_{s,t})^{\sigma\bigstar [p]}_m = H_t(T^{\sigma\bigstar [p]}(A_m,P_m[s])) \ar@{=>}[r] & H_{s+t}(T^{\sigma\bigstar [p]}(A_m,P_m[-_{\ge 0}])).
}
\]
By the induction hypothesis, the induced map
\[
	F_{i,s_{L-i}(m)}^\sigma/F_{i-1,s_{L-i}(m)}^\sigma \to F_{i,m}^{\sigma\bigstar [p_{L-i}]}/F_{i-1,m}^{\sigma\bigstar [p_{L-i}]}
\]
is zero for $1\le i\le L-1$.
Also, observe that $(\sigma\bigstar [a])\bigstar [b] = \sigma\bigstar [(a+1)(b+1)-1]$.
It follows that, by putting $s'(m):=s_1(\dotsb(s_{L-1}(t(m)))\dotsb)$ and $p':=\prod_{l=1}^{L-1}(p_l+1)-1$, the canonical map
\[
	\iota_{s'(m),t(m)}H_l(T^\varphi)\colon H_L(T^\sigma(A_{s'(m)},P_{s'(m)}[-_{\ge 0}])) \to H_L(T^{\sigma\bigstar[p']}(A_{t(m)},P_{t(m)}[-_{\ge 0}]))
\]
factors through $F_{0,t(m)}^{\sigma\bigstar [p']}$.

Now, we have lifts in the commutative diagram
\[
\xymatrix@C-1pc{
	& & H_L(T^\sigma(A_{s(m)},P_{s(m)})) \ar[d]^{\iota_{s(m),s'(m)}} \ar@{.>}[dl] \\
	F_{0,s'(m)}^\sigma \ar@{^{(}->}[r] \ar[d] & H_L(T^\sigma(A_{s'(m)},P_{s'(m)}[-_{\ge 0}])) \ar[r]^-\Theta \ar[d] \ar@{.>}[ld]
		& H_L(T^\sigma(A_{s'(m)},P_{s'(m)})) \ar[d]^{\iota_{s'(m),t(m)}H_L(T^\varphi)} \\
	F_{0,t(m)}^{\sigma\bigstar[p']} \ar@{^{(}->}[r] \ar@/_/@{.>}[d] & H_L(T^{\sigma\bigstar[p']}(A_{t(m)},P_{t(m)}[-_{\ge 0}])) \ar[r]^-\Theta 
		& H_L(T^{\sigma\bigstar[p']}(A_{t(m)},P_{t(m)})) \ar@{=}[d] \\
	H_L(T^{\sigma\bigstar[p']}(A_{t(m)},P_{t(m)}[0])) \ar@{->>}[u] \ar[rr]^-{H_L(T^{\epsilon,\varphi})} & & H_L(T^{\sigma\bigstar[p']}(A_{t(m)},P_{t(m)})).
}
\]
Consider the following diagram
\[
\xymatrix{
	& H_L(T^{\sigma\bigstar[p']}(A_{t(m)},P_{t(m)}[0])) \ar[d] \ar[r]^-{H_L(T^{\epsilon,\varphi})} & H_L(T^{\sigma\bigstar[p]}(A_m,P_m)) \ar@{=}[d] \\
	H_L(T^\sigma(A_{s(m)},P_{s(m)})) \ar[d] \ar[r]_-{H_L(T^\varphi)} \ar@{.>}[ru] & H_L(T^{\sigma\bigstar[p']}(A_{t(m)},P_{t(m)})) \ar[d] \ar[r]_-{H_L(T^\varphi)}
		& H_L(T^{\sigma\bigstar[p]}(A_m,P_m)) \\
	H_L(T^{\sigma_0}(A_{s(m)})) \ar[r]_-{H_L(T^\phi)} & H_L(T^{\sigma_0\times[p']}(A_{t(m)})) \ar[r]_-{H_L(T^{\phi})} & H_L(T^{\sigma_0\times[p]}(A_m)) \ar[u]
}
\]
where we omit the structured maps $\iota_{*,*}$.
The right rectangle commutes by Sublemma \ref{sublem:key}, though the lower right square may not commute.
It follows from a simple diagram chase that the bottom rectangle commutes.
Therefore, by the induction hypothesis for $n=\#\sigma$, the middle composite $\iota_{s(m),m}H_L(T^\varphi)$ equals zero.

Finally, $(ii)_{l=L}$ follows immediately from $(i)_{l=L}$ and Sublemma \ref{sublem:key}.
\end{proof}

\subsubsection{}
Let $\sigma$ be a partial ordering on $\{1,\dotsc,n\}$.
Then we can naturally regard $T^\sigma(A)$ as a subgroup of $\GL_n(A)$.
For $k\ge 0$, we define ${}^k\tilde{\sigma}$ to be the partial ordering on $\{1,\dotsc,n+k\}$ obtained from $\sigma$ by adding the relations $i<n+j$ for $i\in\{1,\dotsc,n\}$ and $1\le j\le k$.
Set
\[
	{}^k\tilde{T}^\sigma(A,P) := \begin{cases} T^\sigma(A) &\text{if }k=0 \\ T^{{}^k\tilde{\sigma}}(A,P) &\text{if }k\ge 1 \end{cases}
							   = T^\sigma(A)\ltimes M_{n,k}(P).
\]
We write $\Pi_n$ for the set of all partial orderings on $\{1,\dotsc,n\}$.

\begin{corollary}\label{cor:key}
Let $\mbf{A}$ be a commutative Tor-unital pro ring and $l,n\ge 0$.
Let $\sigma_1,\dotsc,\sigma_t \in \Pi_n$.
Then there exists $p\ge 0$ such that the canonical map
\[
	\tilde{H}_l\Bigl(\bigcup_{i=1}^t B{}^k\tilde{T}^{\sigma_i}(\mbf{A},\mbf{P})\Bigr)
		\to \tilde{H}_l\Bigl(\bigcup_{i=1}^t B{}^k\tilde{T}^{\sigma_i\times [p]}(\mbf{A},\mbf{P})\Bigr)
\]
is equal to zero as a pro morphism for all $k\ge 0$ and all pseudo-free $\mbf{A}$-modules $\mbf{P}$, where the unions are taken in $B(\GL_n(\mbf{A})\ltimes M_{n,k}(\mbf{P}))$ and $B(\GL_{n(p+1)}(\mbf{A})\ltimes M_{n(p+1),k}(\mbf{P}))$ respectively.

In particular, there exists $N\ge n$ such that the canonical map
\[
	\tilde{H}_l\Bigl(\bigcup_{\sigma\in\Pi_n} B{}^k\tilde{T}^\sigma(\mbf{A},\mbf{P})\Bigr)
		\to \tilde{H}_l\Bigl(\bigcup_{\sigma\in\Pi_N} B{}^k\tilde{T}^\sigma(\mbf{A},\mbf{P})\Bigr)
\]
is equal to zero for all $k\ge 0$ and all pseudo-free $\mbf{A}$-modules $\mbf{P}$.
\end{corollary}
\begin{proof}
Note that
\[
	{}^k\tilde{T}^{\sigma\times [p]}(A,P) 
		= \begin{cases} T^{\sigma\times[p]}(A) &\text{if }k=0 \\ T^{{}^k\tilde{\sigma}\bigstar [p]}(A,P) &\text{if }k\ge 1. \end{cases}
\]
Hence, the case $t=1$ is true by Lemma \ref{lem:key}.
Let $t>1$ and suppose that the corollary holds for $s<t$.

We abbreviate ${}^k\tilde{T}^{\sigma}(\mbf{A},\mbf{P})$ as $\tilde{T}^{\sigma}$.
Set $\sigma_{i,t}:=\sigma_i\cap\sigma_t$.
Then we have a commutative diagram
\[
\xymatrix@C-1pc{
	\tilde{H}_l\bigl(\bigcup_{i=1}^{t-1} B\tilde{T}^{\sigma_i}\bigr) \oplus \tilde{H}_l(B\tilde{T}^{\sigma_t}) \ar[d] \ar[r]
		& \tilde{H}_l\bigl(\bigcup_{i=1}^t B\tilde{T}^{\sigma_i}\bigr) \ar[r] \ar[d] \ar@{.>}[dl] & \tilde{H}_{l-1}\bigl(\bigcup_{i=1}^{t-1} B\tilde{T}^{\sigma_{i,t}}\bigr) \ar[d] \\
	\tilde{H}_l\bigl(\bigcup_{i=1}^{t-1} B\tilde{T}^{\sigma_i\times [q]}\bigr) \oplus \tilde{H}_l(B\tilde{T}^{\sigma_t\times [q]}) \ar[r]
		& \tilde{H}_l\bigl(\bigcup_{i=1}^t B\tilde{T}^{\sigma_i\times [q]}\bigr) \ar[r] & \tilde{H}_{l-1}\bigl(\bigcup_{i=1}^{t-1} B\tilde{T}^{\sigma_{i,t}\times [q]}\bigr)
}
\]
with exact rows.
By the induction hypothesis, the right vertical map is zero for some $q\ge 0$.
Thus, there exists a lift as indicated above.
Again, by the induction hypothesis, there exists $q'\ge 0$ such that the map
\[
	\tilde{H}_l\Bigl(\bigcup_{i=1}^{t-1} B\tilde{T}^{\sigma_i}\Bigr) \oplus \tilde{H}_l(B\tilde{T}^{\sigma_t}) 
		\to \tilde{H}_l\Bigl(\bigcup_{i=1}^{t-1} B\tilde{T}^{\sigma_i\times [q']}\Bigr) \oplus \tilde{H}_l(B\tilde{T}^{\sigma_t\times [q']})
\]
is zero.
It follows from $(\sigma_i\times [q])\times [q'] = \sigma_i\times [q'']$, $q'':=(q+1)(q'+1)-1$, that the map
\[
	\tilde{H}_l\Bigl(\bigcup_{i=1}^t B\tilde{T}^{\sigma_i}\Bigr) \to \tilde{H}_l\Bigl(\bigcup_{i=1}^t B\tilde{T}^{\sigma_i\times [q'']}\Bigr)
\]
is zero.
This completes the proof.
\end{proof}

\subsection{The pro acyclicity theorem}\label{proacyclic}

Recall that ${}^k\tilde{T}^\sigma(A)=T^\sigma(A)\ltimes M_{n,k}(A)$ for $\sigma\in\Pi_n$.

\begin{theorem}\label{thm:proacyclic}
Let $\mbf{A}$ be a commutative Tor-unital pro ring and $l\ge 0$.
Then:
\begin{enumerate}[(i)]
\item For $n\ge 2l+1$ and for any $k\ge 0$,
\[
	\tilde{H}_l\Bigl(\bigcup_{\sigma\in\Pi_n} B{}^k\tilde{T}^\sigma(\mbf{A})\Bigr) = 0,
\]
where the union is taken in $B(\GL_n(\mbf{A})\ltimes M_{n,k}(\mbf{A}))$.
\item For $n\ge 2l$ and for any $k\ge 0$, the canonical map
\[
	H_l\Bigl(\bigcup_{\sigma\in\Pi_n} BT^\sigma(\mbf{A})\Bigr) \to H_l\Bigl(\bigcup_{\sigma\in\Pi_n} B{}^k\tilde{T}^\sigma(\mbf{A})\Bigr)
\]
is a pro isomorphism.
\end{enumerate}
\end{theorem}
\begin{proof}
We write ${}^k\tilde{X}_n(\mbf{A})=\bigcup_{\sigma\in\Pi_n} B{}^k\tilde{T}^\sigma(\mbf{A})$ and $X_n(\mbf{A})={}^0\tilde{X}_n(\mbf{A})$.

We prove the theorem by induction on $l$.
The case $l=0$ is trivial.
Let $L>0$ and suppose that the theorem holds for $l<L$.
\begin{sublemma}\label{sublem:proacyclic}
Let $k\ge 0$.
The canonical map
\[
	H_L({}^k\tilde{X}_n(\mbf{A})) \to H_L({}^k\tilde{X}_{n+1}(\mbf{A}))
\]
is a pro epimorphism for $n\ge 2L$ and a pro isomorphism for $n\ge 2L+1$.
\end{sublemma}
\begin{proof}
Let us introduce some notation.
Let $A$ be a ring, $\sigma\in\Pi_n$ and $1\le i\le n$.
We define $T^{\sigma,i}_n(A)$ be the subgroup of $T^\sigma_n(A)$ consisting of all $\alpha$ with $\alpha_{i,j} = \alpha_{j,i} = 0$ for $i\ne j$.
For $k\ge 0$, we set
\[
	{}^k\tilde{X}^i_n(A) := \bigcup_{\sigma\in\Pi_n}B T^{{}^k\tilde{\sigma},i}(A)
\]
and write ${}^k\tilde{X}^{i_1,\dotsc,i_p}_n(A)$ for the intersection of ${}^k\tilde{X}^{i_1}_n(A),\dotsc,{}^k\tilde{X}^{i_p}_n(A)$.
Then it is easy to see that ${}^k\tilde{X}^{i_1,\dotsc,i_p}_n(A)\simeq {}^k\tilde{X}_{n-p}(A)$.

Consider the spectral sequence
\begin{equation}\label{eq1:proacyclic}
	{}^k\tilde{E}^1_{p,q} = \bigsqcup_{i_0,\dotsc,i_p}H_q\bigl({}^k\tilde{X}^{i_0,\dotsc,i_p}_{n+1}(\mbf{A})\bigr)
		\Rightarrow H_{p+q}\Bigl(\bigcup_{1\le i\le n+1}{}^k\tilde{X}^i_{n+1}(\mbf{A})\Bigr).
\end{equation}
Since ${}^k\tilde{X}^{i_0,\dotsc,i_p}_{n+1}(\mbf{A}) \simeq {}^k\tilde{X}_{n-p}(\mbf{A})$, it follows form the induction hypothesis that 
\[
	{}^k\tilde{E}^2_{0,L} \simeq H_L({}^k\tilde{X}_n(\mbf{A}))
\]
for $n\ge 2L$.
Hence, the canonical map
\[
	H_L({}^k\tilde{X}_n(\mbf{A})) \to H_L\Bigl(\bigcup_{1\le i\le n+1}{}^k\tilde{X}^i_{n+1}(\mbf{A})\Bigr)
\]
is a pro epimorphism for $n\ge 2L$ and a pro isomorphism for $n\ge 2L+1$.
According to \cite[Corollary 6.6, see also the remark before Theorem 7.1]{Su82}\footnotemark, the canonical map
\[
	H_L\Bigl(\bigcup_{1\le i\le {n+1}}{}^k\tilde{X}^i_{n+1}(\mbf{A})\Bigr) \to H_L({}^k\tilde{X}_{n+1}(\mbf{A}))
\]
is a levelwise surjection for $n\ge 2L$ and a levelwise bijection for $n\ge 2L+1$.
\footnotetext{The proof works for non-unital rings as it is.}
Bringing these together, we obtain the sublemma.
\end{proof}

We show $(i)_{l=L}$.
Suppose that $n\ge 2L+1$.
According to Corollary \ref{cor:key}, the canonical map
\[
	H_L({}^k\tilde{X}_n(\mbf{A})) \to H_L({}^k\tilde{X}_N(\mbf{A}))
\]
is zero for some $N\ge n$.
On the other hand, by Sublemma \ref{sublem:proacyclic}, this map is a pro isomorphism,  and thus $H_L({}^k\tilde{X}_n(\mbf{A}))=0$.

To get $(ii)_{l=L}$, it remains to show that the canonical map
\[
	H_L(X_{2L}(\mbf{A})) \to H_L({}^k\tilde{X}_{2L}(\mbf{A}))
\]
is a pro isomorphism.
By the spectral sequences (\ref{eq1:proacyclic}), we have a commutative diagram
\[
\xymatrix{
	0 \ar[r] & E^2_{2,L-1} \ar[r] \ar[d] & H_L(X_{2L}(\mbf{A})) \ar[r] \ar[d] & H_L(X_{2L+1}(\mbf{A})) \ar[r] \ar[d]^\simeq & 0 \\
	0 \ar[r] & {}^k\tilde{E}^2_{2,L-1} \ar[r] & H_L({}^k\tilde{X}_{2L}(\mbf{A})) \ar[r] & H_L({}^k\tilde{X}_{2L+1}(\mbf{A})) \ar[r] & 0
}
\]
with exact rows.
Hence, it is enough to show that $E^2_{2,L-1} \to {}^k\tilde{E}^2_{2,L-1}$ is a pro isomorphism; equivalently it is a pro epimorphism.
This follows from the diagram
\[
\xymatrix{
	E^1_{2,L-1}=\bigoplus H_{L-1}(X_{2L-2}(\mbf{A})) \ar[r] \ar[d]^\simeq & E^2_{2,L-1} \ar[r] \ar[d] & 0 \\
	{}^k\tilde{E}^1_{2,L-1}=\bigoplus H_{L-1}({}^k\tilde{X}_{2L-2}(\mbf{A})) \ar[r] & {}^k\tilde{E}^2_{2,L-1} \ar[r] & 0
}
\]
with exact rows.
\end{proof}

\section{Homology pro stability}\label{MainThm}

In this section, we completes the proof of Theorem \ref{main}.
We follow Suslin \cite{Su96}, generalizing his argument to the pro setting.

\subsection{Volodin spaces}
\label{volodin}

Let $G$ be a group and $\{G_i\}_{i\in I}$ a family of subgroups of $G$.
We define the \textit{Volodin space} $V(G,\{G_i\}_{i\in I})$ to be the simplicial subset of $EG$ formed by simplices $(g_0,\dotsc,g_p)$
such that there exists $i\in I$ and $g_jg_k^{-1}\in G_i$ for all $0\le j,k\le p$.

The simplicial subset $V(G,\{G_i\}_{i\in I}) \subset EG$ is stable under the right action of $G$, and $V(G,\{G_i\})/G=\bigcup_{i\in I}BG_i$.
Hence, we have a spectral sequence
\begin{equation}\label{eq:volodin}
	E^2_{p,q} = H_p(G,H_q(V(G,\{G_i\}_{i\in I}))) \Rightarrow H_{p+q}\Bigl(\bigcup_{i\in I}BG_i\Bigr).
\end{equation}

Let $A$ be a ring.
We consider the Volodin space
\[
	V_n(A) := V(E_n(A),\{T^\sigma(A)\}_{\sigma\in\Pi_n}).
\]
The permutation group $\Sigma_n$ acts on $V_n(A)$ by conjugation, and $E_n(A)$ acts on $V_n(A)$ by the right multiplication.

Here are some properties of the Volodin spaces we need.
\begin{lemma}[Suslin-Wodzicki {\cite[Corollary 2.7]{SW92}}]
\label{lem:volodin}
Let $n\ge 1$ and $k\ge 0$.
The canonical projection and the inclusion
\[
	V_n(A) \rightleftarrows
		V\left(\begin{pmatrix} E_n(A) & * \\ 0 & 1_k \end{pmatrix},\left\{\begin{pmatrix} T^\sigma(A) & * \\ 0 & 1_k \end{pmatrix}\right\}_{\sigma\in\Pi_n}\right)
\]
are mutually inverse homotopy equivalences.
\end{lemma}

\begin{lemma}[Suslin-Wodzick {\cite[Lemma 2.8]{SW92}}]
\label{lem2:volodin}
For every $n,l\ge 0$, the action of $E_{n+1}(A^2)$ on the image of the canonical map
\[
	H_l(V_n(A)) \to H_l(V_{n+1}(A))
\]
is trivial.
\end{lemma}

\begin{corollary}\label{cor:volodin}
Let $\mbf{A}$ be a pro ring such that $\mbf{A}/\mbf{A}^2=0$.
Then, for every $n,l\ge 0$, the action of $E_{n+1}(\mbf{A})$ on the image of the canonical map
\[
	H_l(V_n(\mbf{A})) \to H_l(V_{n+1}(\mbf{A}))
\]
is pro trivial.
\end{corollary}
\begin{proof}
Write $\mbf{A}=\{A_m\}$.
By the assumption, there exists $s\ge m$ for each $m$ such that $\iota_{s,m}A_s\subset A_m^2$.
Hence, given $x$ in the image of $H_l(V_n(A_s)) \to H_l(V_{n+1}(A_s))$ and $g\in E_n(A_s)$, we have $\iota_{s,m}(gx)=\iota_{s,m}(x)$.
\end{proof}

\subsection{Van der Kallen's acyclicity}\label{vdKallen}

Let $A$ be a ring and $n\ge 1$.
Fix a unital ring $R$ which contains $A$ as a two sided ideal.
Let $I$ be a finite subset of $\{1,\dotsc,n\}$ and $R^n$ the free right $R$-module with basis $e_1,\dotsc,e_n$.
A map $f\colon I\to R^n$ is called an \textit{$A$-unimodular function} if $\{f(i)\}_{i\in I}$ forms a basis of a free direct summand of $R^n$ and $f(i)\equiv e_i$ modulo $A$.
We denote by $\msf{Uni}_{A,n}=\msf{Uni}_{A,n}(R)$ the set of all $A$-unimodular functions $f\colon I\to R^n$, which does not depend on $R$.

We define the associated semi-simplicial set as follows:
A $p$-simplex is an $A$-unimodular function $f\in \msf{Uni}_{A,n}$ with $|\dom f|=p+1$.
The $i$-th face $d_i\colon (\msf{Uni}_{A,n})_p \to (\msf{Uni}_{A,n})_{p-1}$, $0\le i\le p$, is defined by 
\[
	(f,\,\dom f=\{i_0,\dotsc,i_p\}) \mapsto f\vert_{\{i_0,\dotsc,\hat{i}_k,\dotsc,i_p\}}.
\]
As in the preceding section, for a semi-simplicial set $X$, we denote by $C_*(X)$ the complex freely generated by $X_*$ with the differential being the alternating sum of the faces.

The following result is proved by van der Kallen \cite{vdK80} in case $A$ is unital, and the proof can be easily modified for non-unital rings.
We can also find the complete proof in \cite[\S2]{Su96}.
\begin{theorem}\label{thm:vdKallen}
$\tilde{H}_l(C_*(\msf{Uni}_{A,n}))=0$ for $n\ge l+\sr(A)+1$.
\end{theorem}

Let $\msf{SUni}_{A,n}$ (resp.\ $\overline{\msf{SUni}}_{A,n}(R)$) be the set of all unimodular functions $f\in \msf{Uni}_{A,n}(R)$
for which there exists $\alpha\in E_n(A)$ (resp.\ $\alpha\in E_n(R,A)$) such that $f(i)=e_i \alpha$ for all $i\in\dom f$.
Then $\msf{SUni}_{A,n}$ and $\overline{\msf{SUni}}_{A,n}(R)$ are sub semi-simplicial sets of $\msf{Uni}_{A,n}(R)$.
\begin{corollary}
\label{cor:vdKallen}
\leavevmode
\begin{enumerate}[(i)]
\item $\tilde{H}_l(C_*(\overline{\msf{SUni}}_{A,n}(R)))=0$ for $n\ge l+\sr(A)+1$.
\item Let $\mbf{A}$ be a pro ring such that $\mbf{A}/\mbf{A}^2=0$.
Then 
\[
	\tilde{H}_l(C_*(\msf{SUni}_{\mbf{A},n})) = 0
\]
as pro abelian groups for $n\ge l+\sr(\mbf{A})+1$.
\end{enumerate}
\end{corollary}
\begin{proof}
(i) See \cite[Corollary 2.8]{Su96}.

(ii) Let $\mbf{R}$ be a unital pro ring which contains $\mbf{A}$ as a two-sided ideal.
By Corollary \ref{cor:Tits}, the canonical map $\msf{SUni}_{\mbf{A},n} \to \overline{\msf{SUni}}_{\mbf{A},n}(\mbf{R})$ is a pro isomorphism.
Hence, (ii) follows from (i).
\end{proof}

\subsection{Homology pro stability for $V_n$ and $E_n$}

We say that a levelwise action of a pro group $\{G_m\}$ on a pro object $\{M_m\}$ is \textit{pro trivial}
if there exists $s\ge m$ for each $m$ such that $\iota_{s,m}(gx)= \iota_{s,m}(x)$ for all $g\in G_s$ and $x\in M_s$.

\begin{theorem}\label{thm:prostability}
Let $\mbf{A}$ be a commutative Tor-unital pro ring.
Let $r=\max(\sr(\mbf{A}),2)$ and $l\ge 0$. Then:
\begin{enumerate}[(i)]
\item The canonical map
\[
	H_l(V_n(\mbf{A})) \to H_l(V_{n+1}(\mbf{A}))
\]
is a pro epimorphism for $n\ge 2l+r+1$ and a pro isomorphism for $n\ge 2l+r+2$.
\item The conjugate action of $\Sigma_n$ on $H_l(V_n(\mbf{A}))$ is pro trivial for $n\ge 2l+r+2$.
\item The action of $E_n(\mbf{A})$ on $H_l(V_n(\mbf{A}))$ is pro trivial for $n\ge 2l+r+2$.
\item The canonical map
\[
	H_l(E_n(\mbf{A})) \to H_l\left(\begin{pmatrix} E_n(\mbf{A}) & * \\ 0 & 1_k \end{pmatrix}\right)
\]
is a pro isomorphism for $n\ge 2l+r-2$ and for any $k\ge 0$.
\item The conjugate action of $\Sigma_n$ on $H_l(E_n(\mbf{A}))$ is pro trivial for $n\ge 2l+r-1$.
\item The canonical map
\[
	H_l(E_n(\mbf{A})) \to H_l(E_{n+1}(\mbf{A}))
\]
is a pro epimorphism for $n\ge 2l+r-2$ and a pro isomorphism for $n\ge 2l+r-1$.
\end{enumerate}
\end{theorem}

We prove Theorem \ref{thm:prostability} by induction on $l$.
The case $l=0$ is clear.
Also, $(iv,v,vi)_{l=1}$ holds for the obvious reasons:
$(v,vi)_{l=1}$ follows from that $H_1(E_n(\mbf{A}))=0$ for $n\ge 3$.
For $(vi)_{l=1}$, note that we have a levelwise exact sequence
\[
\xymatrix@1{
	M_{n,k}(A)_{E_n(\mbf{A})} \ar[r] & H_1(E_n(\mbf{A})\ltimes M_{n,k}(\mbf{A})) \ar[r] & H_1(E_n(\mbf{A})) \ar[r] & 0,
}
\]
and it is easy to see that $M_{n,k}(\mbf{A})_{E_n(\mbf{A})} = 0$ for $n\ge 2$.

Let $L>0$.
The proof is divided into the four steps.
\begin{description}
\item[Step 1:] $(i,ii,iii)_{l<L-1}\Rightarrow (iii)_{l=L-1}$.
\item[Step 2:] $(iii)_{l\le L-1},\, (iv)_{l<L+1} \Rightarrow (iv)_{l=L+1}$.
\item[Step 3:] $(iv)_{l\le L+1},\, (v,vi)_{l<L+1} \Rightarrow (v,vi)_{l=L+1}$.
\item[Step 4:] $(i,ii)_{l<L-1},\, (iii)_{l\le L-1},\, (vi)_{l\le L+1}\Rightarrow (i,ii)_{l=L-1}$.
\end{description}

\subsection{Step 1: Covering argument I}\label{step1}

Suppose that $(i,ii,iii)_{l<L-1}$ hold.
We show $(iii)_{l=L-1}$.
\footnote{In this step, we only need $\Tor_1^{\mbb{Z}\ltimes\mbf{A}}(\mbb{Z},\mbb{Z})=\mbf{A}/\mbf{A}^2=0$.}

\subsubsection{Covering spectral sequence}

Let $A$ be a ring.
For $I\subset \{1,\dotsc,n\}$, let $\Pi^I_n$ be the set of all partial ordering $\sigma$ of $\{1,\dotsc,n\}$ for which every $i\in I$ is maximal.
Set $V_n(A)^I := V_n(E_n(A),\{T^\sigma(A)\}_{\sigma\in\Pi^I_n})$.
Then $V_n(A) = \bigcup_{i=1}^n V_n(A)^i$, and there is a spectral sequence
\[
	E^1_{p,q}(A) = \bigsqcup_{|I|=p+1} H_q(V_n(A)^I) \Rightarrow H_{p+q}(V_n(A)).
\]

Let $\msf{SUni}_{A,n}^I$ be the subset of $\msf{SUni}_{A,n}$ consisting of those functions $f$ with $\dom f=I$.
We define a map $\phi\colon V_n(A)^I \to \msf{SUni}_{A,n}^I$ by $\phi(\alpha_0,\dotsc,\alpha_q)(i)=e_i\alpha_0$, $i\in I$.
Then $\phi$ is a morphism of simplicial sets regarding $\msf{SUni}_{A,n}^I$ as a constant simplicial set, and the inverse image of the unimodular function $f_0\colon i\mapsto e_i$ is $V(E_n(A)^I, \{T^\sigma(A)\}_{\sigma\in \Pi_n^I})$, where $E_n(A)^I$ is the subgroup of $E_n(A)$ generated by elementary matrices $\alpha$ such that $e_i \alpha=e_i$ for all $i\in I$.
For each $f\in \msf{SUni}_{A,n}^I$, choose $\Lambda(f)\in E_n(A)$ with $f(i)=e_i\Lambda(f)$, $i\in I$.
Since the map $\phi$ is $E_n(A)$-equivariant, $\Lambda(f)$ gives an isomorphism $\phi^{-1}(f_0) \simeq \phi^{-1}(f)$ and
\[
	\msf{SUni}_{A,n}^I \times V(E_n(A)^I, \{T^\sigma(A)\}_{\sigma\in \Pi_n^I}) \xrightarrow{\sim} V_n(A)^I, \quad (f,u)\mapsto u\Lambda(f).
\]
Also, the conjugation by the shuffle permutation $\sigma_I$, $\sigma_I\{n-p,\dotsc,n\}=I$, gives an isomorphism
\[
	V(E_n(A)^{n-p,\dotsc,n},\{T^\sigma(A)\}_{\sigma\in \Pi_n^{n-p,\dotsc,n}})\xrightarrow{\sim} V(E_n(A)^I, \{T^\sigma(A)\}_{\sigma\in \Pi_n^I}).
\]
Hence, we get an isomorphism
\[
	\Phi_\Lambda\colon C_p(\msf{SUni}_{A,n})\otimes H_q(V(E_n(A)^{n-p,\dotsc,n},\{T^\sigma(A)\}_{\sigma\in\Pi_n^{n-p,\dotsc,n}}))\xrightarrow{\sim}E^1_{p,q}(A).
\]
For another choice of $\Lambda'$, there exists $\{\gamma(f)\in E_n(A)^{n-p,\dotsc,n}\}_{f\in\msf{SUni}_{A,n}}$ such that $\Phi_{\Lambda'}(f,u)=\Phi_{\Lambda}(f,u\gamma(f))$.

Under the isomorphism $\Phi_\Lambda$, the differential $d^1\colon E_{p,q} \to E_{p-1,q}$ is given by, for $f\in \msf{SUni}_{A,n}^I$ and $u\in H_q(V(E_n(A)^{n-p,\dotsc,n},\{T^\sigma(A)\}_{\sigma\in\Pi_n^{n-p,\dotsc,n}})$,
\begin{equation}\label{eq1:step1}
	d^1 (f\otimes u) = \sum_{k=0}^p (-1)^k d_k f \otimes \tau_{I,k}(\delta u)\tau_{I,k}^{-1}\alpha_k.
\end{equation}
Here, $\alpha_k$ is a certain element in $E_n(A)^{n-p+1,\dotsc,n}$, $\tau_{I,k} := \sigma_{I\setminus\{i_k\}}^{-1}\sigma_I$, and $\delta$ is the map induced from the canonical embedding $E_n(A)^{n-p,\dotsc,n}\to E_n(A)^{n-p+1,\dotsc,n}$.

\subsubsection{Pro arguments}
We write $\mbf{A}=\{A_m\}_{m\in\Xi}$.

Set $\bar{E}_n(\mbf{A}):=\GL_n(\mbf{A})\cap E(\mbf{A})$.
Then the canonical maps
\[
\xymatrix@R-1pc{
	H_q(V(\bar{E}_{n-p-1}(\mbf{A}),\{T^\sigma(\mbf{A})\}_{\sigma\in\Pi_{n-p-1}})) \ar[d]^-\simeq \\
	H_q(V(\bar{E}_n(\mbf{A})^{n-p,\dotsc,n},\{T^\sigma(\mbf{A})\}_{\sigma\in\Pi_n^{n-p,\dotsc,n}})) \ar[d]^-\simeq \\
	H_q\left(V\left({\begin{pmatrix} \bar{E}_{n-p-1}(\mbf{A}) & * \\ 0 & 1_{p+1} \end{pmatrix}},
		\left\{{\begin{pmatrix} T^\sigma(\mbf{A}) & * \\ 0 & 1_{p+1} \end{pmatrix}}\right\}_{\sigma\in\Pi_{n-p-1}}\right)\right)
}
\]
are levelwise isomorphisms.
Indeed, the second map is an isomorphism by definition and the composite is an isomorphism by Lemma \ref{lem:volodin}.
Hence, by Theorem \ref{thm2:stabilityK1}, the canonical map
\[
	\lambda\colon H_q(V_{n-p-1}(\mbf{A})) \to H_q(V(E_n(\mbf{A})^{n-p,\dotsc,n}, \{T^\sigma(\mbf{A})\}_{\sigma\in\Pi_n^{n-p,\dotsc,n}}))
\]
is a pro isomorphism for $n-p-1\ge r+1$.

Suppose that $q<L-1$ and $n-p-1\ge 2q+r+2$.
Then, by $(iii)_{<L-1}$, the action of $E_{n-p-1}(\mbf{A})$ on $H_q(V_{n-p-1}(\mbf{A}))$ is pro trivial.
Hence, there exists $s(m)\ge m$ for each $m\in\Xi$ such that the composite $\Psi_m$
\[
\xymatrix{
	C_p(\msf{SUni}_{A_{s(m)},n})\otimes H_q(V_{n-p-1}(A_{s(m)})) \ar[d]_{\id\otimes\lambda} \ar@{.>}@/^3pc/[rdd]^-{\Psi_m} & \\
	C_p(\msf{SUni}_{A_{s(m)},n})\otimes H_q(V(E_n(A_{s(m)})^{n-p,\dotsc,n},\{T^\sigma(A_{s(m)})\}_{\sigma\in\Pi_n^{n-p,\dotsc,n}}))
		\ar[d]_{\Phi_\Lambda}^\simeq & \\
	E^1_{p,q}(A_{s(m)}) \ar[r]^-{\iota_{s(m),m}} & E^1_{p,q}(A_m)
}
\]
does not depend on the choice of $\Lambda$.
We may assume $s(m+1)>s(m)$ for every $m$, so that we obtain a morphism of pro abelian groups
\[
	\Psi\colon C_p(\msf{SUni}_{\mbf{A},n})\otimes H_q(V_{n-p-1}(\mbf{A})) \to E^1_{p,q}(\mbf{A}).
\]
Since $\lambda$ is a pro isomorphism and $\Phi_\Lambda$ is an isomorphism, we see that $\Psi$ is a pro isomorphism.

Now, by $(ii)_{<L-1}$, the action of $\Sigma_{n-p-1}$ on $H_q(V_{n-p-1}(\mbf{A}))$ is also pro trivial.
Hence, by modifying $s(m)\ge m$ if necessary, we see that the diagram
\[
\xymatrix@C+3pc{
	C_{p+1}(\msf{SUni}_{A_{s(m),n}})\otimes H_q(V_{n-p-2}(A_{s(m)})) \ar[r]^-{\iota_{s(m),m}\Phi_\Lambda(\id\otimes\lambda)} \ar[d]^{\sum(-1)^k d_k \otimes \delta}
		& E^1_{p+1,q}(A_m) \ar[d]^{d^1} \\
	C_p(\msf{SUni}_{A_{s(m),n}})\otimes H_q(V_{n-p-1}(A_{s(m)})) \ar[r]^-{\iota_{s(m),m}\Phi_{\Lambda'}(\id\otimes\lambda)} & E^1_{p,q}(A_m)
}
\]
commutes, cf.\ the formula (\ref{eq1:step1}).
The horizontal maps are the maps $\Psi_m$ unless $n-p-1 = 2q+r+2$; in the last case only the bottom horizontal map can be identified with $\Psi_m$.
Consequently, for $q<L-1$, we obtain a morphism of pro complexes
\begin{equation}\label{eq2:step1}
	\Psi \colon \sigma_{\le n-2q-r-3}(C_\bullet(\msf{SUni}_{\mbf{A},n})\otimes H_q(V_{n-1-\bullet}(\mbf{A}))) \to \sigma_{\le n-2q-r-3}E^1_{\bullet,q}(\mbf{A})\end{equation}
and it is a pro isomorphism.

\begin{claim}\label{claim:step1}
For $q<L-1$ and $0<p\le n-2q-r-3$,
\[
	E^2_{p,q}(\mbf{A}) = 0.
\]
\end{claim}
\begin{proof}
Suppose that $q<L-1$ and $0<p\le n-2q-r-3$.
Put $F_{p,q}(\mbf{A}) := C_p(\msf{SUni}_{\mbf{A},n})\otimes H_q(V_{n-p-1}(\mbf{A}))$, which we regard as a complex in $p$ with the differential $\partial:=\sum(-1)^k d_k \otimes \delta$.
First, we show that $H_p(F_{\bullet,q}(\mbf{A})) = 0$.

By $(i)_{<L-1}$, the canonical map $H_q(V_{n-p-1}(\mbf{A})) \to H_q(V_{n-p}(\mbf{A}))$ is a pro isomorphism, and thus
\[
	\ker(F_{p,q}(\mbf{A}) \to F_{p-1,q}(\mbf{A})) \simeq Z_p(\msf{SUni}_{\mbf{A},n}) \otimes H_q(V_{n-p-1}(\mbf{A})),
\]
where $Z_p(\msf{SUni}_{\mbf{A},n}) :=\ker(C_p(\msf{SUni}_{\mbf{A},n})\to C_{p-1}(\msf{SUni}_{\mbf{A},n}))$.
According to Corollary \ref{cor:vdKallen}, the differential
\[
	C_{p+1}(\msf{SUni}_{\mbf{A},n}) \to Z_p(\msf{SUni}_{\mbf{A},n})
\]
is a pro epimorphism.
Also, by $(i)_{<L-1}$, the canonical map
\[
	H_q(V_{n-p-2}(\mbf{A})) \to H_q(V_{n-p-1}(\mbf{A}))
\]
is a pro epimorphism.
These imply that $\partial\colon F_{p+1}(\mbf{A})\to\ker(F_{p,q}(\mbf{A}) \to F_{p-1,q}(\mbf{A}))$ is a pro epimorphism, hence $H_p(F_{\bullet,q}(\mbf{A})) = 0$.

If $p<n-2L-r-3$, then $\Psi$ (\ref{eq2:step1}) induces a pro isomorphism
\[
	H_pF_{\bullet,q}(\mbf{A}) \simeq E^2_{p,q}(\mbf{A}).
\]
Hence, in this case, the vanishing of $E^2_{p,q}(\mbf{A})$ follows from the one of $H_p(F_{\bullet,q}(\mbf{A}))$.

Finally, let $p=n-2q-r-3$.
Then we have a commutative diagram
\[
\xymatrix@C+3pc{
	F_{p+1,q}(A_{s(m)})\ar[r]^-{\iota_{s(m),m}\Phi_\Lambda(\id\otimes\lambda)} \ar[d]^\partial & E^1_{p+1,q}(A_m)\ar[d]^{d^1} \\
	F_{p,q}(A_{s(m)}) \ar[r]^-{\Psi_m}							  							   & E^1_{p,q}(A_m).
}
\]
Let $m'\ge m$ and $x\in \ker(E^1_{p,q}(A_{m'})\to E^1_{p-1,q}(A_{m'}))$.
Since $\Psi$ is a pro isomorphism, if we have taken $m'$ large enough, $\iota_{m',m}x$ lifts to $y\in\ker(F_{p,q}(A_{s(m)})\to F_{p-1,q}(A_{s(m)}))$ along $\Psi_m$.
Further, since $H_{p,q}(F_{\bullet,q}(\mbf{A}))=0$, we may assume that $y=\partial z$ for some $z\in F_{p+1,q}(A_{s(m)})$.
Hence, $\iota_{m',m}x$ is in the image of the differential $d^1$.
This proves $E^2_{p,q}(\mbf{A})=0$.
\end{proof}

\subsubsection{Conclusion}
Suppose that $n\ge 2L+r$.
If $p+q=L-1$ and $p>0$, then $q<L-1$ and $0<p\le n-2q-r-3$.
Hence, by Claim \ref{claim:step1}, the $E^2_{p,q}$-terms with $p+q=L-1$ are zero unless $E^2_{0,L-1}$, and the edge map
\[
	E^1_{0,L-1}(\mbf{A}) \to H_{L-1}(V_n(\mbf{A}))
\]
is a pro epimorphism.

Now, the composite
\[
\xymatrix{
	C_0(\msf{SUni}_{A_m,n})\otimes H_{L-1}(V_{n-1}(A_m)) \ar[d]_{\id\otimes\lambda} \ar@{.>}@/^2pc/[rdd] & \\
	C_0(\msf{SUni}_{A_m,n})\otimes H_{L-1}(V(E_n(A_m)^{\{n\}},\{T^\sigma(A_m)\}_{\sigma\in\Pi_n^{\{n\}}})) \ar[d]_{\Phi_\Lambda} & \\
	E^1_{0,L-1}(A_m) \ar[r]^-{\text{edge}} & H_{L-1}(V_n(A_m))
}
\]
is given by $f\otimes u \mapsto \sigma_{\{i\}}(\delta u)\sigma_{\{i\}}^{-1}\Lambda(f)$, where $f$ is an $A_m$-unimodular function with $\dom f=\{i\}$ and $u\in H_{L-1}(V_{n-1}(A_m))$.
Since the action of $E_n(\mbf{A})$ on the image of $\delta\colon H_{L-1}(V_{n-1}(\mbf{A}))\to H_{L-1}(V_n(\mbf{A}))$ is pro trivial by Corollary \ref{cor:volodin}, the above composite yields a pro morphism
\begin{equation}\label{eq3:step1}
	C_0(\msf{SUni}_{\mbf{A},n})\otimes H_{L-1}(V_{n-1}(\mbf{A})) \to H_{L-1}(V_n(\mbf{A})), \quad f\otimes u \mapsto \sigma_{\{i\}}(\delta u)\sigma_{\{i\}}^{-1}.
\end{equation}
Furthermore, since the edge map is a pro epimorphism, $\Phi_\Lambda$ is an isomorphism and $(\id\otimes\lambda)$ is a pro isomorphism, we see that (\ref{eq3:step1}) is a pro epimorphism.

By Corollary \ref{cor:volodin} again, we conclude that the action of $E_n(\mbf{A})$ on $H_{L-1}(V_n(\mbf{A}))$ is pro trivial.
This proves $(iii)_{l=L-1}$.

\subsection{Step 2: $V$ to $E$}

Suppose that $(iii)_{l\le L-1}$ and $(vi)_{l<L+1}$ hold.
We show $(vi)_{l=L+1}$.

Suppose that $n\ge 2L+r$ and fix $k\ge 0$.
We set 
\[
	\tilde{E}_n(\mbf{A}) := \begin{pmatrix} E_n(\mbf{A}) & * \\ 0 & 1_k \end{pmatrix}, \quad
	\tilde{T}^\sigma(\mbf{A}) := \begin{pmatrix} T^\sigma(\mbf{A}) & * \\ 0 & 1_k \end{pmatrix}
\]
and $\tilde{V}_n(\mbf{A}) := V(\tilde{E}_n(\mbf{A}),\{\tilde{T}^\sigma(\mbf{A})\}_{\sigma\in\Pi_n})$.

By Lemma \ref{lem:volodin}, the canonical inclusion and projection $V_n(\mbf{A})\rightleftarrows\tilde{V}_n(\mbf{A})$ are mutually inverse homotopy equivalences.
It follows that the action of $\left(\begin{smallmatrix} 1_n & * \\ 0 & 1_k \end{smallmatrix}\right)$ on $H_*(\tilde{V}_n(\mbf{A}))$ is trivial.
By $(iii)_{\le L-1}$, the action of $E_n(\mbf{A})$ on $H_q(V_n(\mbf{A}))\simeq H_q(\tilde{V}_n(\mbf{A}))$ is pro trivial for $q\le L-1$.
Hence, the action of $\tilde{E}_n(\mbf{A})$ on $H_q(\tilde{V}_n(\mbf{A}))$ is pro trivial for $q\le L-1$.

Consider the spectral sequences (\ref{eq:volodin}) and the canonical map between them;
\[
\xymatrix@C-1pc@R-1pc{
	E^2_{p,q}(\mbf{A})= H_p(E_n(\mbf{A}),H_q(V_n(\mbf{A}))) \ar@{=>}[r] \ar[d] 					& H_{p+q}(\bigcup_{\sigma\in\Pi_n} BT^\sigma(\mbf{A})) \\
	\tilde{E}^2_{p,q}(\mbf{A})= H_p(\tilde{E}_n(\mbf{A}),H_q(\tilde{V}_n(\mbf{A}))) \ar@{=>}[r] & H_{p+q}(\bigcup_{\sigma\in\Pi_n} B\tilde{T}^\sigma(\mbf{A})).
}
\]
For $q\le L-1$, the $E^2$-terms fit into the extensions
\[
\xymatrix@C-1pc{
	0 \ar[r] & H_p(E_n(\mbf{A}))\otimes H_q(V_n(\mbf{A})) \ar[r] \ar[d]
		& E^2_{p,q}(\mbf{A}) \ar[r] \ar[d] & \Tor(H_{p-1}(E_n(\mbf{A})),H_q(V_n(\mbf{A}))) \ar[r] \ar[d] & 0 \\
	0 \ar[r] & H_p(\tilde{E}_n(\mbf{A}))\otimes H_q(\tilde{V}_n(\mbf{A})) \ar[r] 
		& \tilde{E}^2_{p,q}(\mbf{A}) \ar[r] & \Tor(H_{p-1}(\tilde{E}_n(\mbf{A})),H_q(\tilde{V}_n(\mbf{A}))) \ar[r] & 0.
}
\]
By $(iv)_{<L+1}$, the canonical map $H_p(E_n(\mbf{A}))\to \tilde{H}_p(E_n(\mbf{A}))$ is a pro isomorphism for $p\le L$.
Hence, the canonical map
\[
	E^2_{p,q}(\mbf{A}) \to \tilde{E}^2_{p,q}(\mbf{A})
\]
is a pro isomorphism for $p\le L$ and $q\le L-1$.
Also, $E^2_{0,q}(\mbf{A}) \simeq \tilde{E}^2_{0,q}(\mbf{A})$ for all $q\ge 0$, since $H_*(V_n(\mbf{A})) \simeq H_*(\tilde{V}_n(\mbf{A}))$.
Finally, by Theorem \ref{thm:proacyclic}, the canonical map $E^\infty_i(\mbf{A}) \to \tilde{E}^\infty_i(\mbf{A})$ is a pro isomorphism for $n\ge 2i$.

Bringing these together, we have:
\begin{enumerate}[(1)]
\item $E^2_{p,q}(\mbf{A}) \simeq \tilde{E}^2_{p,q}(\mbf{A})$ for $p+q=L-1$.
\item $E^2_{p,q}(\mbf{A}) \simeq \tilde{E}^2_{p,q}(\mbf{A})$ for $p+q=L$.
\item $E^2_{p,q}(\mbf{A}) \simeq \tilde{E}^2_{p,q}(\mbf{A})$ for $p+q=L+1$ and $p\ge 2$ and $q\ge 1$.
\item $E^\infty_L(\mbf{A}) \simeq \tilde{E}^\infty_L(\mbf{A})$ and $E^\infty_{L+1}(\mbf{A}) \simeq \tilde{E}^\infty_{L+1}(\mbf{A})$.
\end{enumerate}
Then, by Lemma \ref{lem:spectral1} below, we conclude that
\[
	E^2_{L+1,0}(\mbf{A}) \to \tilde{E}^2_{L+1,0}(\mbf{A})
\]
is a pro epimorphism, and thus a pro isomorphism.
This proves $(vi)_{l=L+1}$.

\begin{lemma}[{\cite[Remark A.5]{Su96}}]\label{lem:spectral1}
Let $\mcal{A}$ be an abelian category.
Let $f\colon E\to \tilde{E}$ be a morphism of first quadrant homological spectral sequence in $\mcal{A}$, and let $L\ge 0$.
Assume that $f$ induces:
\begin{enumerate}[(1)]
\item A monomorphism $E^2_{p,q} \hookrightarrow \tilde{E}^{2}_{p,q}$ for $p+q=L-1$.
\item An isomorphism $E^2_{p,q} \xrightarrow{\sim} \tilde{E}^2_{p,q}$ for $p+q=L$.
\item An epimorphism $E^2_{p,q} \twoheadrightarrow \tilde{E}^2_{p,q}$ for $p+q=L+1$, $q\ge 1$ and $p\ge 2$.
\item An isomorphism $E^\infty_L \xrightarrow{\sim} \tilde{E}^\infty_L$ and an epimorphism $E^\infty_{L+1} \twoheadrightarrow \tilde{E}^\infty_{L+1}$.
\end{enumerate}
Then $f$ induces an epimorphism
\[
	E^2_{L+1,0} \twoheadrightarrow \tilde{E}^2_{L+1,0}.
\]
\end{lemma}

\subsection{Step 3: Covering argument II}

Suppose that $(iv)_{l\le L+1}$ and $(v,vi)_{l<L+1}$ hold.
We show $(v,vi)_{l=L+1}$.

\begin{sublemma}\label{sublem:step3}
For $l\le L+1$ and $n\ge 2l+r-2$, the conjugate action of $\GL_{n+1}(\mbb{Z})$ on the image of
\[
	H_l(E_n(\mbf{A})) \to H_l(E_{n+1}(\mbf{A}))
\]
is pro trivial.
\end{sublemma}
\begin{proof}
The case $l=0,1$ is clear.
Suppose that $2\le l\le L+1$ and $n\ge 2l+r-2$.

Since $\GL_{n+1}(\mbb{Z})=\mbb{Z}\times\SL_{n+1}(\mbb{Z})$ and $\SL_{n+1}(\mbb{Z}) =E_{n+1}(\mbb{Z})$, $\GL_{n+1}(\mbb{Z})$ is generated by $e_{i,n+1}(1)$, $e_{n+1,i}(1)$, $1\le i\le n$, and $\diag(1,\dotsc,1,-1)$.
It is obvious that $\diag(1,\dotsc,1,-1)$ acts trivially on the image of $H_l(E_n(\mbf{A})) \to H_l(E_{n+1}(\mbf{A}))$.

We show the triviality of the conjugate action of $e_{i,n+1}(1)$; the one of $e_{n+1,i}(1)$ is similar.
By Corollary \ref{cor:Tits}, it suffices to show that the action on the image of 
\[
	H_l(E_n(\mbf{R},\mbf{A})) \to H_l(E_{n+1}(\mbf{R},\mbf{A}))
\]
is pro trivial for some unital pro ring $\mbf{R}$ which contains $\mbf{A}$ as a two-sided ideal.
The inclusion $E_n(\mbf{R},\mbf{A})\hookrightarrow E_{n+1}(\mbf{R},\mbf{A})$ factors through  
\[
	\tilde{E}_n(\mbf{R},\mbf{A}):= \begin{pmatrix} E_n(\mbf{R},\mbf{A}) & * \\ 0 & 1 \end{pmatrix} \subset E_{n+1}(\mbf{R},\mbf{A})
\]
and it is normalized by $e_{i,n+1}(1)$.
Hence, it suffices to show that $e_{i,n+1}(1)$ acts pro trivially on the image of $H_l(E_n(\mbf{R},\mbf{A})) \to H_l(\tilde{E}_n(\mbf{R},\mbf{A}))$.
Now, we have a commutative diagram
\[
\xymatrix@C+1pc{
	H_l(\tilde{E}_n(\mbf{R},\mbf{A})) \ar[r]^{e_{i,n+1}(1)} & H_l(\tilde{E}_n(\mbf{R},\mbf{A})) \ar@{->>}[d] \\
	H_l(E_n(\mbf{R},\mbf{A})) \ar@{^{(}->}[u] \ar[r]^\id 	& H_l(E_n(\mbf{R},\mbf{A}),
}
\]
and the vertical maps, the canonical inclusion and projection, are pro isomorphisms by $(iv)_{\le L+1}$.
This implies that $e_{i,n+1}(1)$ acts pro trivially on the image of $H_l(E_n(\mbf{R},\mbf{A})) \to H_l(\tilde{E}_n(\mbf{R},\mbf{A}))$.
\end{proof}

We consider the hyperhomology spectral sequence
\[
	E^1_{p,q}(\mbf{A}) = H_q(E_{n+1}(A),C_p(\msf{SUni}_{\mbf{A},n+1})) \Rightarrow H_{p+q}(E_{n+1}(\mbf{A}),C_\bullet(\msf{SUni}_{\mbf{A},n+1})).
\]
The $C_p(\msf{SUni}_{\mbf{A},n+1})$ decomposes into a direct sum of $E_{n+1}(A)$-submodules $C_p(\msf{SUni}_{\mbf{A},n+1}^I)$ with $|I|=p+1$, and we have a levelwise isomorphism $\mbb{Z}E_{n+1}(\mbf{A})\otimes_{\mbb{Z}E_{n+1}(\mbf{A})^I}\mbb{Z} \xrightarrow{\sim} C_p(\msf{SUni}_{\mbf{A},n+1})^I$, which sends $\alpha\in E_{n+1}(\mbf{A})$ to the unimodular function $i\mapsto e_i\alpha$, $i\in I$.
Hence,
\[
	\bigsqcup_{|I|=p+1} H_q(E_{n+1}(\mbf{A})^I) \simeq E^1_{p,q}(\mbf{A}).
\]

Let $\Delta^n$ be the nerve of the partially ordered set $\{1<2<\dotsb<n+1\}$.
We define level maps $E_{n-p}(\mbf{A})\to E_{n+1}(\mbf{A})^I$ by sending $\alpha$ to $\sigma_I\left(\begin{smallmatrix}\alpha & 0 \\ 0 & 1_{p+1}\end{smallmatrix}\right)\sigma_I^{-1}$, where $\sigma_I$ is the shuffle permutation $\sigma_I\{n-p+1,\dotsc,n+1\}=I$.
These maps yield
\[
	\Psi\colon \Delta^n_p\otimes H_q(E_{n-p}(\mbf{A}))\simeq \bigsqcup_{|I|=p+1}H_q(E_{n-p}(\mbf{A}))
		\to \bigsqcup_{|I|=p+1} H_q(E_{n+1}(\mbf{A})^I) \simeq E^1_{p,q}(\mbf{A}).
\]
It follows from Theorem \ref{thm2:stabilityK1} and $(iv)_{\le L+1}$ that $\Psi$ is a pro isomorphism for $q\le L+1$ and $n-p\ge\max(2q+r-2,r+1)$.
Furthermore, by Sublemma \ref{sublem:step3}, we see that the diagram
\[
\xymatrix{
	\Delta^n_{p+1}\otimes H_q(E_{n-p-1}(\mbf{A})) \ar[d]_{\sum_{k=0}^{p+1}(-1)^kd_k\otimes\delta} \ar[r]^-\Psi
		& E^1_{p+1,q}(\mbf{A}) \ar[d]^{d^1} \\
	\Delta^n_p\otimes H_q(E_{n-p}(\mbf{A})) \ar[r]^-\Psi & E^1_{p,q}(\mbf{A})
}
\]
commutes for $q\le L+1$ and $n-p\ge 2q+r-1$, where $d_k$ are the face maps of $\Delta^n$ and $\delta$ is the canonical map $H_q(E_{n-p-1}(\mbf{A}))\to H_q(E_{n-p}(\mbf{A}))$.

\begin{claim}\label{claim:step3}
For $q\le L$ and $0<p\le n-2q-r+1$,
\[
	E^2_{p,q}(\mbf{A}) = 0.
\]
\end{claim}
\begin{proof}
Suppose that $q\le L$ and $0<p\le n-2q-r+1$.
Put $F_{p,q}(\mbf{A}) := \Delta^n_p\otimes H_q(E_{n-p}(\mbf{A}))$, which we regard as a complex in $p$ with differential $\sum_{k=0}^{p+1}(-1)^kd_k\otimes\delta$.
Then, by $(vi)_{<L+1}$, we have
\[
	\ker(F_{p,q}(\mbf{A}) \to F_{p-1,q}(\mbf{A})) \simeq \ker(\mbb{Z}\Delta^n_p\to \mbb{Z}\Delta^n_{p-1}) \otimes H_q(E_{n-p}(\mbf{A})).
\]
Again by $(vi)_{<L+1}$, the canonical map
\[
	H_q(E_{n-p-1}(\mbf{A})) \to H_q(E_{n-p}(\mbf{A}))
\]
is a pro epimorphism.
Since $\Delta^n$ is contractible, we conclude that $H_p(F_{\bullet,q}(\mbf{A}))=0$.

Now, we have a pro isomorphism
\[
	E^2_{p,q}(\mbf{A}) \simeq H_p(F_{\bullet,q}(\mbf{A}))
\]
for $n-p-1\ge r+1$.
Our assumption says $n-p-1\ge 2q+r-2$; hence, in case $2q+r-2\ge r+1$, the vanishing of $E^2_{p,q}(\mbf{A})$ follows form the one of $H_p(F_{\bullet,q}(\mbf{A}))$.

It remains to show the case $q=1$.
However, in this case,
\[
	E^1_{p,1}(\mbf{A}) \xrightarrow[\Psi]{\sim} \Delta^n\otimes H_1 (E_{n-p}(\mbf{A})) = 0.
\]
This completes the proof.
\end{proof}

Suppose that $n\ge 2L+r$.
Then the $E^2$-terms with $p+q=L+1$ are zero unless $E^2_{0,L+1}(\mbf{A})$.
Hence, the edge map
\[
	E^1_{0,L+1}(\mbf{A}) \to E^\infty_{L+1}(\mbf{A})
\]
is a pro epimorphism.
The left hand side is pro isomorphic to $\Delta^n_0\otimes H_{L+1}(E_n(\mbf{A}))$ by $\Psi$.
According to Corollary \ref{cor:vdKallen}, $\tilde{H}_i(C_*(\msf{SUni}_{\mbf{A},n+1})) = 0$ for $n\ge i+r$.
Hence, we have a pro isomorphism
\[
	E^\infty_{L+1}(\mbf{A})=H_{L+1}(E_{n+1}(\mbf{A}),C_\bullet(\msf{SUni}_{\mbf{A},n+1})) \simeq H_{L+1}(E_{n+1}(\mbf{A})).
\]
By using Sublemma \ref{sublem:step3}, we see that the edge map
\[
	\Delta^n_0\otimes H_{L+1}(E_n(\mbf{A})) \to H_{L+1}(E_{n+1}(\mbf{A}))
\]
coincides as a pro morphism with the sum of the canonical map $\delta\colon H_{L+1}(E_n(\mbf{A})) \to H_{L+1}(E_{n+1}(\mbf{A}))$.
Hence, the $\delta$ is a pro epimorphism.
This proves the first half of $(vi)_{l=L+1}$.

Next, suppose that $n\ge 2L+r+1$.
Then by Claim \ref{claim:step3}, $E^s_{s,L-s+2}(\mbf{A}) = 0$ for $s\ge 2$.
Hence, we have an exact sequence
\[
\xymatrix@1{
	\Delta^n_1\otimes H_{L+1}(E_{n-1}(\mbf{A})) \ar[r] & \Delta^n_0\otimes H_{L+1}(E_n(\mbf{A})) \ar[r] & H_{L+1}(E_{n+1}(\mbf{A})) \ar[r] & 0.
}
\]
Since $H_{L+1}(E_{n-1}(\mbf{A})) \to H_{L+1}(E_n(\mbf{A}))$ is a pro epimorphism, we conclude that the canonical map
\[
	H_{L+1}(E_n(\mbf{A})) \xrightarrow{\sim} H_{L+1}(E_{n+1}(\mbf{A}))
\]
is a pro isomorphism.
This proves the second half of $(vi)_{l=L+1}$.

Finally, since $H_{L+1}(E_{n-1}(\mbf{A})) \to H_{L+1}(E_n(\mbf{A}))$ is a pro epimorphism, Sublemma \ref{sublem:step3} implies that the action of $\Sigma_n$ on $H_{L+1}(E_n(\mbf{A}))$ is pro trivial.
This proves $(v)_{l=L+1}$

\subsection{Step 4: $E$ to $V$}\label{step4}

Suppose that $(i,ii)_{l<L-1}$, $(iii)_{l\le L-1}$ and $(vi)_{l\le L+1}$ hold.
We show $(i,ii)_{l\le L-1}$.

Suppose that $n\ge 2L+r$.
Consider the spectral sequences (\ref{eq:volodin}) and the canonical morphism between them;
\[
\xymatrix@C-1pc@R-1pc{
	{}^nE^2_{p,q}(\mbf{A}) = H_p(E_n(\mbf{A}),H_q(V_n(\mbf{A}))) \ar@{=>}[r] \ar[d] 	 & H_{p+q}\bigl(\bigcup_{\sigma\in\Pi_n} BT^\sigma(\mbf{A})\bigr) \\
	{}^{n+1}E^2_{p,q}(\mbf{A}) = H_p(E_{n+1}(\mbf{A}),H_q(V_{n+1}(\mbf{A}))) \ar@{=>}[r] & H_{p+q}\bigl(\bigcup_{\sigma\in\Pi_{n+1}} BT^\sigma(\mbf{A})\bigr).
}
\]
By $(iii)_{\le L-1}$, for $q\le L-1$, the $E^2$-terms fit into the extensions
\[
\xymatrix@C-1pc{
	0 \ar[r] & H_p(E_n(\mbf{A}))\otimes H_q(V_n(A^m)) \ar[r] \ar[d] 
		& {}^nE^2_{p,q}(\mbf{A}) \ar[r] \ar[d] & \Tor(H_{p-1}(E_n(\mbf{A})),H_q(V_n(\mbf{A}))) \ar[r] \ar[d] & 0 \\
	0 \ar[r] & H_p(E_{n+1}(\mbf{A}))\otimes H_q(V_{n+1}(\mbf{A})) \ar[r] 
		& {}^{n+1}E^2_{p,q}(\mbf{A}) \ar[r] & \Tor(H_{p-1}(E_{n+1}(\mbf{A})),H_q(V_{n+1}(\mbf{A}))) \ar[r] & 0.
}
\]
Hence, it follows from $(i)_{<L-1}$ and $(vi)_{\le L+1}$ that the map 
\[
	{}^nE^2_{p,q}(\mbf{A}) \to {}^{n+1}E^2_{p,q}(\mbf{A})
\]
is a pro epimorphism for $q<L-1$ and $p\le L+1$, and it is a pro isomorphism if further $n\ge 2p+r-1$.
Finally, by Theorem \ref{thm:proacyclic}, ${}^nE^\infty_i(\mbf{A}) \simeq {}^{n+1}E^\infty_i(\mbf{A})$ for $n\ge 2i+1$.

Bringing these together, we have:
\begin{enumerate}[(1)]
\item ${}^nE^2_{p,q}(\mbf{A}) \simeq {}^{n+1}E^2_{p,q}(\mbf{A})$ for $p+q=L-1$ and $p\ge 1$.
\item ${}^nE^2_{p,q}(\mbf{A}) \simeq {}^{n+1}E^2_{p,q}(\mbf{A})$ for $p+q=L$ and $p\ge 2$.
\item ${}^nE^2_{p,q}(\mbf{A}) \twoheadrightarrow {}^{n+1}E^2_{p,q}(\mbf{A})$ for $p+q=L+1$ and $p\ge 3$.
\item ${}^nE^\infty_{L-1}(\mbf{A}) \simeq {}^{n+1}E^\infty_{L-1}(\mbf{A})$ and ${}^nE^\infty_L(\mbf{A}) \simeq {}^{n+1}E^\infty_L(\mbf{A})$.
\end{enumerate}
Then, by Lemma \ref{lem:spectral2} below, we conclude that the canonical map
\[
	{}^nE^2_{0,L-1}(\mbf{A}) \xrightarrow{\sim} {}^{n+1}E^2_{0,L-1}(\mbf{A})
\]
is a pro isomorphism.
By $(iii)_{\le L-1}$, the left hand side (resp.\ right hand side) is pro isomorphic to $H_{L-1}(V_n(\mbf{A}))$ (resp.\ $H_{L-1}(V_{n+1}(\mbf{A}))$).
Hence, we get the second part of $(i)_{l=L-1}$.

Next, we show $(ii)_{l=L-1}$.
Now, the canonical map
\[
	H_{L-1}(V_n(\mbf{A})) \xrightarrow{\sim} H_{L-1}(V_{n+2}(\mbf{A}))
\]
is a $\Sigma_n$-equivariant pro isomorphism.
Hence, it suffices to show that $\Sigma_{n+2}$ (and thus $\Sigma_n$) acts pro trivially on $H_{L-1}(V_{n+2}(\mbf{A}))$.
Now, the permutation $\tau_{n+1,n+2}$ acts pro trivially on $H_{L-1}(V_{n+2}(\mbf{A}))$, since it acts trivially on the image of the above map.
Since $\Sigma_{n+2}$ is the normal closure of $\tau_{n+1,n+2}$, $\Sigma_{n+2}$ also acts pro trivially on $H_{L-1}(V_{n+2}(\mbf{A}))$.

In Step 1, we have seen that the map (\ref{eq3:step1})
\[
	C_0(\msf{SUni}_{\mbf{A},n})\otimes H_{L-1}(V_{n-1}(\mbf{A})) \to H_{L-1}(V_n(\mbf{A}))
\]
sending $f\otimes u \mapsto \sigma_{\{i\}} (\delta u) \sigma_{\{i\}}^{-1}$ ($\dom f= \{i\}$) is a pro epimorphism for $n\ge 2L+r$.
Now, we know that $\sigma_{\{i\}} (\delta u) \sigma_{\{i\}}^{-1} = \delta u$.
Therefore, the canonical map $\delta\colon H_{L-1}(V_{n-1}(\mbf{A})) \to H_{L-1}(V_n(\mbf{A}))$ is a pro epimorphism.
This completes the proof of $(i)_{l=L-1}$.

\begin{lemma}[{\cite[Theorem A.6]{Su96}}]\label{lem:spectral2}
Let $\mcal{A}$ be an abelian category.
Let $f\colon E\to \tilde{E}$ be a morphism of first quadrant homological spectral sequence in $\mcal{A}$, and let $L>0$.
Assume that $f$ induces:
\begin{enumerate}[(1)]
\item A monomorphism $E^2_{p,q} \hookrightarrow \tilde{E}^{2}_{p,q}$ for $p+q=L-1$, $p\ge 1$.
\item An isomorphism $E^2_{p,q} \xrightarrow{\sim} \tilde{E}^2_{p,q}$ for $p+q=L$, $p\ge 2$.
\item An epimorphism $E^2_{p,q} \twoheadrightarrow \tilde{E}^2_{p,q}$ for $p+q=L+1$, $p\ge 3$.
\item Isomorphisms $E^\infty_{L-1} \xrightarrow{\sim} \tilde{E}^\infty_{L-1}$ and $E^\infty_L \xrightarrow{\sim} \tilde{E}^\infty_L$.
\end{enumerate}
Then $f$ induces an isomorphism
\[
	E^2_{0,L-1} \xrightarrow{\simeq} \tilde{E}^2_{0,L-1}.
\]
\end{lemma}

\subsection{Homology pro stability for $\GL_n$}

Now, we can prove Theorem \ref{main}.
We restate it here.
\begin{theorem}\label{thm2:prostability}
Let $\mbf{A}$ be a commutative Tor-unital pro ring.
Let $r=\max(\sr(\mbf{A}),2)$ and $l\ge 0$.
Then the canonical map
\[
	H_l(\GL_n(\mbf{A})) \to H_l(\GL_{n+1}(\mbf{A}))
\]
is a pro epimorphism for $n\ge 2l+r-2$ and a pro isomorphism for $n\ge 2l+r-1$.
\end{theorem}
\begin{proof}
The case $l=0$ is clear.
The case $l=1$ is proved in Theorem \ref{thm:stabilityK1}.
Let $l\ge 2$ and $n\ge 2l+r-2$.
Then, by Theorem \ref{thm:stabilityK1} and Corollary \ref{cor:Tits}, the sequence
\[
\xymatrix@1{
	0 \ar[r] & E_n(\mbf{A}) \ar[r] & \GL_n(\mbf{A}) \ar[r] & H_1(\GL(\mbf{A})) \ar[r] & 0.
}
\]
is exact up to pro isomorphisms.
Now, we have a morphism of spectral sequences;
\[
\xymatrix@C-1pc@R-1pc{
	{}^nE_{p,q}^2 = H_p(H_1(\GL(\mbf{A})),H_q(E_n(\mbf{A}))) \ar@{=>}[r] \ar[d]  & H_{p+q}(\GL_n(\mbf{A})) \ar[d] \\
	{}^{n+1}E_{p,q}^2 = H_p(H_1(\GL(\mbf{A})),H_q(E_{n+1}(\mbf{A}))) \ar@{=>}[r] & H_{p+q}(\GL_{n+1}(\mbf{A})).
}
\]
Using these spectral sequences, we can easily deduce the theorem from Theorem \ref{thm:prostability} (vi).
\end{proof}

\begin{corollary}\label{cor:protrivial}
Let $\mbf{B}$ be a pro ring with a two-sided ideal $\mbf{A}$ and $r=\max(\sr(\mbf{A}),2)$.
Suppose that $\mbf{A}$ is commutative and Tor-unital.
Then the conjugate action of $\GL_n(\mbf{B})$ on $H_l(\GL_n(\mbf{A}))$ is pro trivial for $n\ge 2l+r-1$.
\end{corollary}
\begin{proof}
Let $\alpha$ (resp.\ $\beta$) be the map $\GL_n\to\GL_{2n}$ given by
\[
	g \mapsto \begin{pmatrix}g & 0\\ 0 & 1_n\end{pmatrix} \quad \text{resp.\ }g \mapsto \begin{pmatrix}1_n & 0\\ 0 & g\end{pmatrix}.
\]
According to the Theorem \ref{thm2:prostability}, the induced maps
\[
	\alpha,\beta\colon H_l(\GL_n(\mbf{A})) \xrightarrow{\sim} H_l(\GL_{2n}(\mbf{A}))
\]
are pro isomorphisms for $n\ge 2l+r-1$.

Write $\mbf{B}=\{B_m\}_{m\in J}$ and $\mbf{A}=\{A_m\}_{m\in J}$.
For each $m\in J$, choose $s(m)\ge m$ such that if $\alpha(a)=0$ with $a\in H_l(\GL_n(A_{s(m)}))$ then $\iota_{s(m),m}(a)=0$.
Next, choose $t(m)\ge s(m)$ such that for every $x\in H_l(\GL_n(A_{t(m)}))$ there exists $y\in H_l(\GL_n(A_{s(m)}))$ with $\iota_{t(m),s(m)}(\alpha(x)) = \beta(y)$.
Then, for $g\in\GL_n(B_{t(m)})$ and $x\in H_l(\GL_n(A_{t(m)}))$,
\[
\begin{split}
	\alpha(\iota_{t(m),s(m)}(gx)) &= \alpha(\iota_{t(m),s(m)}(g))\beta(y)  \\
								  &= \beta(y) \\
								  &= \alpha(\iota_{t(m),s(m)}(x)).
\end{split}
\]
Hence, $\iota_{t(m),m}(gx) = \iota_{t(m),m}(x)$.
This completes the proof.
\end{proof}

Suslin has shown that if a ring $A$ is Tor-unital then for every ring $B$ which contains $A$ as a two-sided ideal the conjugate action of $\GL(B)$ on $H_l(\GL(A))$ is trivial, cf.\ \cite[Corollary 4.5]{Su95}, see also \cite[Corollary 1.6]{SW92}.
Geisser-Hesselholt has generalized Suslin's result to a pro setting, cf.\ \cite[Proposition 1.3]{GH06}.
Here is a straightforward generalization of their result.
\begin{theorem}[Suslin, Geisser-Hesselholt]\label{thm:SGH}
Let $\mbf{B}$ be a pro ring with a two-sided ideal $\mbf{A}$.
Suppose that $\mbf{A}$ is Tor-unital.
Then the conjugate action of $\GL(\mbf{B})$ on $H_l(\GL(\mbf{A}))$ is pro trivial for all $l\ge 0$.
\end{theorem}

By using Theorem \ref{thm:SGH}, we can strengthen our main theorem, Theorem \ref{thm2:prostability}.

\begin{theorem}
Let $\mbf{A}$ be a commutative Tor-unital pro ring, $r=\max(\sr(\mbf{A}),2)$ and $l\ge 0$.
Suppose that there exists a unital pro ring $\mbf{R}$ with $\sr(\mbf{R})<\infty$ which contains $\mbf{A}$ as a two-sided ideal.
Then the canonical map
\[
	H_l(\GL_n(\mbf{A})) \to H_l(\GL(\mbf{A}))
\]
is a pro epimorphism for $n\ge 2l+r-2$ and a pro isomorphism for $n\ge 2l+r-1$.
\end{theorem}
\begin{proof}
Let $\mbf{R}$ be a unital pro ring as in the statement.
Consider the commutative diagram
\[
\xymatrix{
	\GL_n(\mbf{A}) \ar[r] \ar[d] & \GL_n(\mbf{R}) \ar[r] \ar[d] & \overline{\GL}_n(\mbf{R}/\mbf{A}) \ar[d] \\
	\GL(\mbf{A}) \ar[r] 		 & \GL(\mbf{R}) \ar[r] 			& \overline{\GL}(\mbf{R}/\mbf{A})
}
\]
with exact rows.
Now, the second and the third maps induce isomorphisms on homology for $n$ large enough.
Also, the action of $\GL_n(\mbf{R})$ on $H_l(\GL_n(\mbf{A}))$ is pro trivial for $n$ large enough (Theorem \ref{cor:protrivial}) and for $n=\infty$ (Theorem \ref{thm:SGH}).
Consequently, the canonical map
\[
	H_l(\GL_n(\mbf{A})) \xrightarrow{\sim} H_l(\GL(\mbf{A}))
\]
is a pro isomorphism for $n$ large enough.
Combining it with Theorem \ref{thm2:prostability}, we get the result.
\end{proof}

\end{document}